\documentclass[11pt]{article} 
\usepackage{color,cite,enumerate}
\usepackage{bm}
\usepackage{amsmath} 
\usepackage{amsthm} 
\usepackage{amssymb}	
\usepackage{graphicx} 
\usepackage{multicol} 
\usepackage{amsfonts}
\usepackage{mathrsfs}
\usepackage[top=1.0in, bottom=1.0in, left=1.0in, right=1.0in]{geometry}
\usepackage{epstopdf}
\usepackage{upgreek}
\usepackage{cleveref}
\usepackage{authblk,tikz}

\def\Xint#1{\mathchoice
   {\XXint\displaystyle\textstyle{#1}}%
   {\XXint\textstyle\scriptstyle{#1}}%
   {\XXint\scriptstyle\scriptscriptstyle{#1}}%
   {\XXint\scriptscriptstyle\scriptscriptstyle{#1}}%
   \!\int}
\def\XXint#1#2#3{{\setbox0=\hbox{$#1{#2#3}{\int}$}
     \vcenter{\hbox{$#2#3$}}\kern-.5\wd0}}

\def\dashint{\Xint-}

\numberwithin{equation}{section} 

\newcounter{marnote}

\usepackage{soul}

\newcommand{\brackets}[1]{\left( #1 \right)}

\newcommand{\abs}[1]{\left| #1 \right|} 

\renewcommand{\d}[1]{\, \mathrm{d} #1} 
\newcommand{\pd}[2]{\frac{\partial #1}{\partial #2}}

\newcommand{\grad}[1]{\nabla #1} 
\renewcommand{\div}[1] {\nabla \cdot #1} 

\newcommand{\beq}{\begin{equation}}
\newcommand{\eeq}{\end{equation}}
\newcommand{\R}{\mathbb{R}}

\newcommand{\eps}{\varepsilon}

\let\baraccent=\= 

\renewcommand{\=}[1]{\stackrel{#1}{=}} 

\newtheorem{theorem}{Theorem}[section]
\newtheorem{lemma}{Lemma}[section]
\theoremstyle{definition}

\theoremstyle{Theorem}
\newtheorem{rmk}{Remark}[section]
\DeclareMathOperator{\diam}{diam}
\DeclareMathOperator{\dist}{dist}

\title{\bf Edge domain walls in ultrathin exchange-biased films}

\author[1]{Ross G. Lund} 
\author[1]{Cyrill B. Muratov}
\author[2]{Valeriy V. Slastikov}

\affil[1]{\em Department of Mathematical Sciences, New Jersey
    Institute of Technology, Newark, NJ 07102, USA} 
\affil[2]{\em School of Mathematics, University of Bristol,
    Bristol BS8 1TW, UK}

\begin{document}
\maketitle

\begin{abstract}
  We present an analysis of edge domain walls in exchange-biased
  ferromagnetic films appearing as a result of a competition between
  the stray field at the film edges and the exchange bias field in the
  bulk. We introduce an effective two-dimensional micromagnetic energy
  that governs the magnetization behavior in exchange-biased materials
  and investigate its energy minimizers in the strip geometry. In a
  periodic setting, we provide a complete characterization of global
  energy minimizers corresponding to edge domain walls. In particular,
  we show that energy minimizers are one-dimensional and do not
  exhibit winding.  We then consider a particular thin film regime for
  large samples and relatively strong exchange bias and derive a
  simple and comprehensive algebraic model describing the limiting
  magnetization behavior in the interior and at the boundary of the
  sample. Finally, we demonstrate that the asymptotic results obtained
  in the periodic setting remain true in the case of finite
  rectangular samples.
\end{abstract}

\section{Introduction}
\label{sec:introduction}


Ferromagnetic films and multilayers are fundamental nanostructures
widely used in present day magnetoelectronics devices \cite{prinz98}.
As such, they have been the subject of intensive investigations over
the last two decades in the engineering, physics and applied
mathematics communities
\cite{hubert,bader10,dennis02,fidler00,desimone06r}. Some of the
highlights of these activities include the discoveries of giant
magnetoresistance, spin-transfer torque, spin-orbit coupling and the
spin-Hall effect \cite{bader10,brataas12, soumyanarayanan16,
  hellman17}. These new physical phenomena have lead to the design of
such technological applications as magnetic sensors, actuators,
high-density magnetic storage devices and non-volatile computer
memory.

Surface and interfacial effects play a dominant role and are
responsible for determining many properties of the nanostructured
ferromagnetic materials \cite{hubert,dennis02,hellman17}. These
phenomena become increasingly important in the case of ultrathin films
and multilayers. One basic example of such nanostructures is given by
exchange-biased materials, which consist of a ferromagnetic film on
top of an antiferromagnetic layer \cite{nogues05}. As a consequence of
an exchange coupling between the two layers, the magnetization in the
ferromagnetic film experiences a net bias induced by the magnetization
at the interlayer interface, which furnishes the free layer with an
effective unidirectional anisotropy.  Additionally, nanostructure
edges may also drastically change the equilibrium and the dynamic
behaviors of the magnetization.  For instance, the nanostructure edges
often determine the mechanism of the magnetization reversal process
\cite{hubert, e03,mo:jap08}. However, despite the importance of edge
effects there exist just a handful of rigorous analytical studies
characterizing the magnetization behavior near the film edges
\cite{kohn05arma, kurzke06, moser04, lms:non18, ms:prsla16}.

Formation of {\it edge domain walls} is an important manifestation of
edge effects observed in ferromagnetic films, double layers and
exchange-biased materials \cite{hornreich63,hornreich64,wade64,
  ruhrig90,mattheis97,cho99,hubert,dennis02,rebousas09}. Edge domain
walls appear as the result of a competition between magnetostatic
energy dominating near the edges and the anisotropy or bias field
effects in the bulk, leading to a mismatch in the preferred
magnetization directions near and far from the film edges. It is well
known that in ultrathin ferromagnetic films without perpendicular
magnetic anisotropy the magnetization prefers to stay almost entirely
in the film plane.  At the same time, the magnetization tends to stay
parallel to the film edge even if the magnetocrystalline anisotropy or
the bias field favor a different magnetization direction in the
interior. This effect is due to the stray field energy which produces
a significant contribution near the sample edges
\cite{kohn05arma}. Inside the sample, the bias field and/or
magnetocrystalline anisotropy dominate the micromagnetic energy,
favoring a single domain state. When these effects are sufficiently
strong, they may also influence the magnetization behavior close the
sample boundary.  As a result of the competition between the stray
field and anisotropy/exchange bias energies, also taking into account
the exchange energy, a transition layer near the edge, called edge
domain wall, is formed. Although this simple phenomenological
explanation gives an intuitive picture, apart from a few ansatz-based
studies in the physics literature \cite{hornreich63, hornreich64,
  nonaka85, hirono86} there is currently little quantitative
understanding of this phenomenon.

The goal of this paper is to understand the formation of edge domain
walls in exchange-biased materials, viewed as minimizers of the
micromagnetic energy. We are interested in soft ultrathin
ferromagnetic films in the presence of a strong exchange bias
field. Our analysis is based on a reduced two-dimensional
micromagnetic energy with magnetization vector constrained to lie in
the film plane, which is well known to adequately describe the
magnetization behavior in ultrathin ferromagnetic films
\cite{mo:jcp06,kohn05arma,desimone06r}. Since we are concerned with
the magnetization behavior near the edges, we consider one of the
simplest and yet application relevant geometries, namely, that of a
ferromagnetic strip. As described earlier, in this geometry the
magnetization inside the strip aligns with the direction of the bias
field, but at the edges it tends to align along the fixed edge
direction. Typically, there is a misalignment between these two
directions which, with the help of the exchange energy, results in the
formation of a boundary layer near the edge (see
Fig. \ref{fig:stripnum}). Let us stress that the situation considered
here is very different from the case treated in \cite{kohn05arma},
where the magnetization behavior at the boundary is controlled by the
magnetization in the interior through the trace theorem.  In larger
ferromagnetic samples considered here the exchange energy does not
impose enough control over magnetization variation. This results in
the detachment of the trace of the interior magnetization profile from
the magnetization at the sample boundary. In particular, the actual
magnetization behavior at the boundary is determined in a non-trivial
way through the competition of exchange bias, stray field and bulk
exchange energies.

\begin{figure}
  \centering
  \includegraphics[width=14cm]{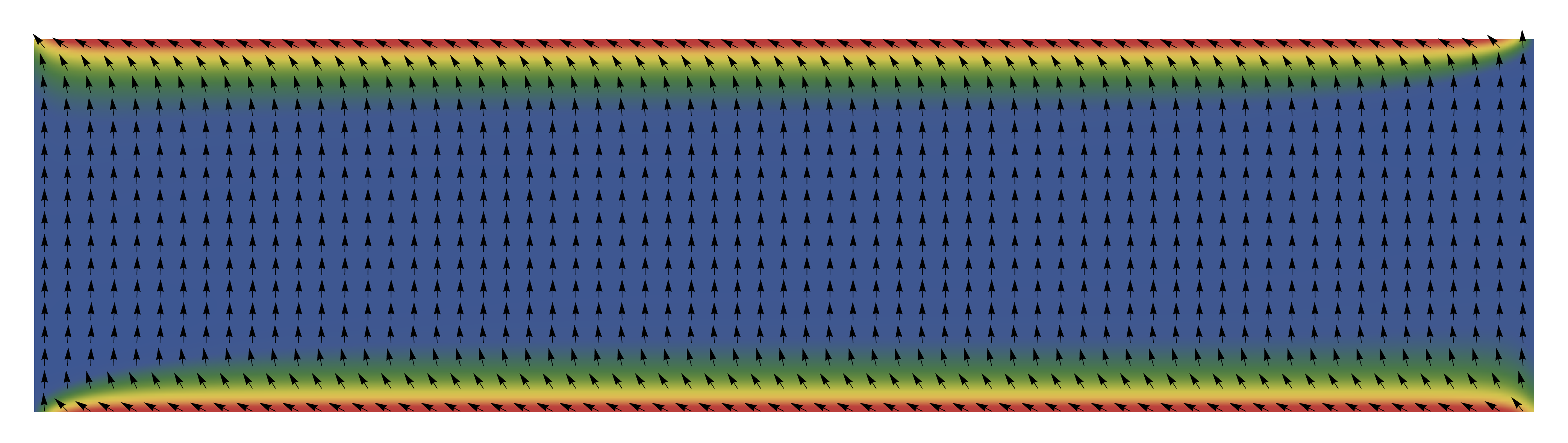}
  \caption{A remanent magnetization in an exchange-biased permalloy
    film (exchange constant $A = 1.3 \times 10^{-11}$ J/m, saturation
    magnetization $M_s = 8 \times 10^5$ A/m, exchange bias field
    $H = 8.91 \times 10^3$ A/m) with dimensions
    $3.46\mu$m$\times 0.87\mu$m$\times 6$nm. Result of a micromagnetic
    simulation, using the code developed in \cite{mo:jcp06}. The bias
    field is pointing up. Edge domain walls exhibiting partial
    alignment of the magnetization with the sample edges may be seen
    at the top and the bottom boundary.}
  \label{fig:stripnum}
\end{figure}


Our analysis of the above problem in nanomagnetism proceeds as
follows. First, we introduce a two-dimensional model, see \eqref{E},
which governs the magnetization behavior in exchange-biased ultrathin
nanostructures and accounts for the presence of nanostructure
edges. This model is an extension of a reduced thin film model
introduced in the context of Ginzburg-Landau systems with dipolar
repulsion that provides matching upper and lower bounds on the full
three-dimensional energy for vanishing film thickness, together with
universal error estimates \cite{m:cvar}. Instead of treating the
magnetization as a discontinuous vector field having length one inside
and zero outside a three-dimensional sample, we consider a
two-dimensional domain occupied by the film in the plane (viewed from
the top) and introduce a narrow band near the film edge, comparable in
size to the film thickness. In this band the magnetization is
regularized for the stray field calculation, using a smooth cut-off
function, see \eqref{mdelta}. Note that the magnetization behavior is
asymptotically independent of the choice of the cutoff.  We then
proceed to analyse global energy minimizers associated with the energy
in \eqref{E} in the presence of strong exchange bias in the direction
normal to the strip edge.

We point out that the obtained non-convex, non-local, vectorial
variational problem in full generality poses a formidable challenge to
analysis. In particular, the system under consideration is known to
exhibit winding magnetization configurations \cite{cho99}, which
further complicates the situation. Nevertheless, within a periodic
setting we are able to provide a complete characterization of global
energy minimizers of the energy in \eqref{E}. We first show that the
energy minimizing configurations are one-dimensional, i.e., in those
configurations the magnetization depends only on the distance to the
edges. Furthermore, the magnetization vector does not exhibit winding
and may rotate by at most 90 degrees away from the bias field
direction. Thus, in the periodic setting the task of globally
minimizing the energy \eqref{E} reduces to a particular
one-dimensional variational problem. For the latter, we prove that
there exist at most three minimizers, which are smooth solutions to a
non-local Euler-Lagrange equation and possess $C^2$ regularity up to
boundary, see Theorem \ref{t:1d}.

We then consider a particular thin film regime, in which the sample
lateral dimensions also go to infinity with an appropriate rate, while
the exchange bias, bulk exchange and magnetostatic energies all
balance near the strip edge, see \eqref{Rescaled}, \eqref{lamdeleps}
and \eqref{hepslnln}. Still within the periodic setting, we then
derive a simple and comprehensive {\it algebraic} model describing the
magnetization behavior in the interior and at the boundary of the
ferromagnet in the regime of strong exchange bias in the limit as the
film thickness goes to zero, see Theorem \ref{t:E1deps}. This reduced
model uniquely determines the magnetization trace at the film edge for
the minimizers, see Theorem \ref{t:1dmin}. We also show that after a
blowup the magnetization profile near the edge converges uniformly to
an explicit profile in \eqref{thinf}. Finally, we demonstrate that the
asymptotic results for the limit behavior of the energy and the
average trace of the magnetization on the sample edges obtained in the
periodic setting remain true in the case of rectangular domains, see
Theorem \ref{t:rect}.

Our proofs in the periodic setting rely on a sharp, strict lower bound
for the energy in \eqref{E} of a two-dimensional magnetization
configuration in terms of the energy in \eqref{E1dm2} evaluated on the
averages along the direction of the strip of the component of the
magnetization normal to the strip edge. For the magnetostatic part of
the energy, the corresponding lower bound is obtained, using Fourier
techniques. For the local part of the energy, we use its convexity as
a function of that component in the absence of winding. The latter is
ensured by the choice of the reconstruction of the magnetization
vector from the average of its component in the direction normal to
the edge. We note that this argument crucially uses the specific form
of the exchange bias energy and does not apply in the case of the
uniaxial anisotropy considered by us in \cite{lms:non18}. Once the
one-dimensional nature of the minimizers has been established, the
derivation of the Euler-Lagrange equation and the regularity still
requires a delicate analysis due to the fact that nonlocality remains
intertwined with the rest of the terms, producing an
integro-differential equation. Additionally, under our Lipschitz
assumption on the cutoff function, which also allows to mimic films
with tapering edges, the non-local term may produce singularities near
the sample boundaries, limiting the regularity of the minimizers up to
the boundary.  Finally, using the Euler-Lagrange equation we are able
to show that the tangential component of the magnetization in a
minimizer does not change sign. This allows us to take advantage of
the convexity of the one-dimensional energy as a function of the
normal component under this condition to establish the precise
multiplicity of the minimizers.

For our asymptotic analysis, we first remark that in our problem it is
necessary to go beyond the magnetostatics contribution at the sample
edges considered in \cite{kohn05arma}. Indeed, since the magnetization
in the sample interior converges to a constant vector, the net
magnetic line charge density at the strip edges is constant to the
leading order. Therefore, one needs to perform an asymptotic expansion
to extract the leading order non-trivial contribution associated with
the charge distribution between the strip edge and the strip interior
in the boundary layer near the edge. After subtracting the leading
order constant, we deduce the asymptotic behavior of the minimal
energy and the energy minimizers by establishing matching asymptotic
upper and lower bounds on the energy. The lower bounds are a
combination of the Modica-Mortola type bounds for the local part of
the energy, while for the magnetostatic energy we use carefully chosen
test potentials in a duality formulation that goes back to Brown
\cite{brown}. In turn, the upper bounds rely on explicit
Modica-Mortola transition layer profiles with an optimized boundary
trace. Finally, we show that the presence of the additional edges
parallel to the bias direction does not affect the asymptotic behavior
of the energy for rectangular samples.

Our paper is organized as follows. In Sec. \ref{sec:model}, we present
the two-dimensional model analyzed throughout the paper and discuss
the relevant scaling regime. In Sec. \ref{sec:main-results}, we state
our main results. In Sec. \ref{sec:proof-theor-reft:1d}, we present
the proof of Theorem \ref{t:1d} that characterizes the energy
minimizers in the periodic setting. In Sec. \ref{sec:proof-theor-23},
we present the proofs of Theorems \ref{t:E1deps} and \ref{t:1dmin}
about the asymptotic behavior of the minimizers in the periodic
setting in the considered regime. Finally, in Sec. \ref{sec:rect}, we
present the proof of Theorem \ref{t:rect} about the asymptotics of the
minimizers on a rectangular domain.

\paragraph{Acknowledgements.}

R. G. Lund and C. B. Muratov were supported, in part, by NSF via
grants DMS-1313687 and DMS-1614948.  V. Slastikov would like to
acknowledge support from EPSRC grant EP/K02390X/1 and Leverhulme grant
RPG-2014-226.

\section{Model}
\label{sec:model}

In this paper we investigate ultrathin ferromagnetic films with
negligible magnetocrystalline anisotropy and in the presence of an
exchange bias, which manifests itself as a Zeeman-like term in the
energy. As our films of interest are only a few atomic layers thin, it
is appropriate to model them using a two-dimensional micromagnetic
framework. Furthermore, in the absence of perpendicular magnetic
anisotropy the equilibrium magnetization vector is constrained to lie
almost entirely in the film plane
\cite{desimone00,kohn05,garcia01,mo:jcp06}. Therefore, in the case of
an extended film the magnetization state may be described by a map
$m : \R^2 \to \mathbb S^1$, with the associated energy (after a
suitable rescaling) given by
\begin{align}
  \label{ER2}
  E(m) = \frac12 \int_{\R^2} \Big( |\nabla m|^2 + h | m - e_2 |^2 
  \Big) \d x + {\delta \over 8 \pi} \int_{\R^2} \int_{\R^2}
  {\div{m}(x) \div{m}(y) \over |x - y|} \d x \d y. 
\end{align}
Here, the terms in the order of appearance are: the exchange energy,
the Zeeman-like exchange bias energy due to an adjacent fixed magnetic
layer, and the stray field energy, respectively
\cite{hubert,desimone00,nogues05}. In writing \eqref{ER2} we
measured lengths in the units of the exchange length and introduced
the {\em effective} dimensionless film thickness $\delta > 0$ that
plays the role of the strength of the magnetostatic interaction. Also,
we have introduced the dimensionless constant $h > 0$ that
characterizes the strength of the exchange bias along the vector
$e_2$, the unit vector in the direction of the second coordinate
axis. Note that due to rotational symmetry of the exchange and
magnetostatic energies, the choice of the direction in the exchange
bias term is arbitrary. Observe that by positive definiteness of the
stray field term the unique global minimizer for the energy in
\eqref{ER2} is given by the monodomain state $m(x) = e_2$.

\subsection{Energy of a finite sample}

We now turn our attention to films of finite extent, i.e., when the
ferromagnetic material occupies a bounded domain in the plane,
$D\subset \R^2$. One would naturally expect that the above model can
be easily modified to describe the finite sample case by restricting
the domains of integration to $D$. However, this is not the case as
such a model would miss the contribution of the edge charges to the
magnetostatic energy \cite{kohn05arma}. On the other hand, a simple
extension of the magnetization $m$ from $D$ to the whole of $\R^2$ by
zero and treating $\nabla \cdot m$ distributionally would not work in
general, as in this case the magnetostatic energy becomes infinite
unless the magnetization is tangential to the boundary $\partial D$ of
the sample (for further discussion see \cite{lms:non18}). This is due
to the fact that a discontinuity in the normal component of the
magnetization at the sample edge produces a divergent contribution to
the magnetostatic energy. Physically, however, the magnetization near
the edges of the film is smooth on the atomic scale, which for
ultrathin films is comparable to the film thickness
$\delta$. Therefore, we can introduce a regularization of the
magnetization \beq \label{mdelta0} m_\delta (x) := \eta_\delta (x)
m(x) \qquad x \in D, \eeq where
\begin{align}
  \label{mdelta}
  \eta_\delta(x):= \eta\brackets{\frac{\dist(x, \partial D)}{\delta}},
\end{align}
and $\eta \in C^\infty(\overline{\R}^+)$ satisfies $\eta'(t) > 0$ for
all $0 < t < 1$, $\eta(0)=0$ and $\eta(t)=1$ for all $t \geq 1$.  This
defines a Lipschitz cutoff at scale $\delta$ near $\partial D$ to smear the
film edge on the scale of its thickness. The precise choice of the
cutoff function will be unimportant.  The two-dimensional
micromagnetic energy modelling the ultrathin ferromagnetic film of
finite extent is now defined as
\begin{align}
  \label{E}
  E(m) =
  \frac{1}{2} \int_{D} \brackets{\abs{\grad {m}}^2 + h\abs{m-e_2}^2
  }\d x + \frac{\delta}{8\pi} \int_{D}\int_{D} \frac{\div m_\delta(x)
  \div m_\delta(y)}{\abs{x-y}}\d x \d y.
\end{align}
This energy is the starting point of our investigation. 

\subsection{Energy in a periodic setting}
\label{sec:periodic}

We are also interested in a particular situation in which the domain
has the shape of an infinite strip along the $x_1$ direction, of width
$b > 0$; this situation is not immediately covered by the previous
discussion. We assume periodicity in $x_1$ with period $a > 0$ and
define the energy per period:
\begin{align}
  \label{Esh}
  E^\#(m) =
  \frac{1}{2} \int_{D} \brackets{\abs{\grad {m}}^2 + h\abs{m-e_2}^2
  }\d x + \frac{\delta}{8\pi} \int_{D}\int_{\R \times (0,b)}
  \frac{\div m_\delta(x) \div m_\delta(y)}{\abs{x-y}}\d y \d x,
\end{align}
where $D = (0,a) \times (0,b)$. Note that this energy is
translationally invariant in the $x_1$-direction. In particular,
one-dimensional magnetization configurations independent of $x_1$ are
natural candidates for minimizers of $E^\#$. We point out that
choosing the strip axis to lie along the direction $e_1$
(perpendicular to $e_2$) creates a competition between the exchange
bias favoring $m$ to lie along $e_2$ and the shape anisotropy forcing
$m$ to lie along $e_1$, which makes this configuration the most
interesting one.

\subsection{Connection to three-dimensional micromagnetics}
\label{sec:connection-3-d}

Let us point out that the energy in \eqref{E} may also be justified in
some regimes by considering suitable thin film limits of the full
three-dimensional micromagnetic energy \beq \mathcal{E}(\mathbf m) =
\frac12\int_{\Omega} \brackets{\abs{\grad {\mathbf m}}^2 + h |\mathbf
  m - e_2|^2} \d x + \frac{1}{8\pi} \int_{\R^3}\int_{\R^3} \frac{\div
  \mathbf m(x) \div \mathbf m(y)}{\abs{x-y}}\d x \d y,
\label{Energy}
\eeq where $\Omega \subset \R^3$ is the domain occupied by the
material and $\mathbf m : \Omega \to \mathbb S^2$, with $\mathbf m$
extended by zero outside $\Omega$ and $\nabla \cdot m$ understood
distributionally. Typically when considering thin films the domain
$\Omega$ is taken to be a cylinder $\Omega = D \times (0,\delta)$,
where $D \subset \R^2$ is the base of the film and $\delta$ is the
film thickness \cite{desimone00}. In reality the film edges are never
straight, but vary on the scale of the film thickness $\delta$, and
averaging over the thickness we recover an analogue of the regularized
magnetization $m_\delta$ introduced in \eqref{mdelta0} (for further
discussion in a related context, see \cite{m:cvar}). Indeed, when
$0 < \delta \lesssim 1$, the out-of-plane component of the
magnetization $\mathbf m(x) \in \mathbb S^2$ is strongly penalized,
forcing the magnetization to be restricted to the equator of
$\mathbb S^2$, identified with $\mathbb S^1$. Furthermore, the
magnetization vector will be effectively constant on the length scale
$\delta$. Therefore, to the leading order in $\delta$ we will have
\begin{align}
  \label{mS2S1}
  \mathbf m(x_1, x_2, x_3) = (m(x_1, x_2), 0) \qquad
  m : \R^2 \to \mathbb S^1,   
\end{align}
and
$\mathcal E(\mathbf m) \simeq E(m) \delta$, where $m_\delta$ in
\eqref{E} is defined, using a cutoff function $\eta_\delta$ related to
the shape of the film edge (see also \cite{slastikov05}).

\subsection{Thin film regime}
\label{sec:thin-film-regime}

We now introduce a particular asymptotic regime in which edge domain
walls bifurcate from the monodomain state $m = e_2$ as global energy
minimizers when the effective film thickness $\delta \to 0$.  We note
that for all other parameters fixed the minimizer of the
two-dimensional energy in \eqref{E} or the three-dimensional energy in
\eqref{Energy} would converge to the monodomain state (for a closely
related result, see \cite{m:cvar}). Therefore, in order to observe
non-trivial minimizers in the thin film limit the lateral size of the
ferromagnetic sample must diverge with an appropriate rate
simultaneously with $\delta \to 0$. To capture this balance, we
introduce a small parameter $\eps > 0$ corresponding to the inverse
lateral size of the ferromagnetic sample, i.e.,
$\diam(D_\eps) = O(\eps^{-1})$ and set $\delta = \delta_\eps \to 0$ as
$\eps \to 0$. We also allow $h = h_\eps$ to depend on $\eps$. We then
have a one-parameter family of functionals, parametrized by $\eps$ and
given by $E(m) = E_\eps^0(m)$, where \beq E^0_\eps(m) = \frac{1}{2}
\int_{D_\eps} \brackets{\abs{\grad {m}}^2 + h_\eps \abs{m-e_2}^2 }\d x
+ \frac{\delta_\eps}{8\pi} \int_{\R^2}\int_{\R^2} \frac{\div
  m_{\delta_\eps}(x) \div m_{\delta_\eps}(y)}{\abs{x-y}}\d x \d y,
\eeq with a slight abuse of notation, assuming the cutoff function in
\eqref{mdelta} is defined, using $D_\eps$ instead of $D$.  If we then
rescale $D_\eps$ to work on an $O(1)$ domain $D$, we obtain that
$E^0_\eps(m) = \eps^{-1} E_\eps(m(\cdot / \eps))$, where \beq
E_\eps(m) := \frac{1}{2} \int_{D} \brackets{\eps\abs{\grad {m}}^2 +
  \frac{h_\eps}{\eps}\abs{m-e_2}^2}\d x + \frac{\delta_\eps}{8\pi}
\int_{\R^2}\int_{\R^2} \frac{\div m_{\eps \delta_\eps}(x) \div m_{\eps
    \delta_\eps}(y)}{\abs{x-y}}\d x \d y.
\label{Rescaled}
\eeq 

To proceed, we take, once again, the domain $D$ to be a rectangle,
$D = (0, a) \times (0, b)$, and consider two magnetization
configurations as competitors. The first one is the monodomain state
$m^{(1)} = e_2$ and the second one is a profile $m^{(2)}$ in which the
magnetization rotates smoothly from $e_2$ in
$(0, a) \times (\eps L_\eps, b - \eps L_\eps)$ to $e_1$ at $x_2 = 0$
and $x_2 = b$ within layers of width $\eps L_\eps$ near the top and
bottom edges of $D$ such that
$\eps \delta_\eps \ll \eps L_\eps \ll 1$. Note that while in the
former the edge magnetic charges are concentrated within layers of
thickness $\delta_\eps$ (in the original, unscaled variables), in the
latter the edge magnetic charges are spread within layers of width
$L_\eps$ (again, before rescaling).

It is not difficult to see that as $\eps \to 0$ we have
\begin{align}
  \label{Eepsm1m2}
  E_\eps(m^{(1)}) \simeq {a \delta_\eps \over 2 \pi} | \ln 
  \eps \delta_\eps|, \qquad E_\eps(m^{(2)}) \simeq a \left( {c_1
  \over
  L_\eps} + c_2 h_\eps L_\eps 
  \right) +  {a \delta_\eps \over 2 \pi} |\ln 
  \eps L_\eps|,
\end{align}
for some $c_{1,2} > 0$ depending on the choice of the transition
profile.  Clearly, when the exchange bias field $h_\eps = O(1)$ the
first two terms give an $O(1)$ contribution to the energy
$E_\eps(m^{(2)})$. Therefore, in order for the energy of the edge
charges $E(m^{(1)})$ in a monodomain state to be comparable with the
local contributions to the energy of edge domain walls one needs to
choose
\begin{align}
  \label{lamdeleps}
  \delta _\eps = {\lambda \over |\ln \eps|},
\end{align}
for some $\lambda > 0$ playing the role of the renormalized effective
film thickness. Notice that this scaling has recently appeared in a
different context in the studies of thin ferromagnetic films with
perpendicular magnetic anisotropy \cite{kmn:arma}. At the same time,
according to \eqref{Eepsm1m2} the leading order contribution to the
magnetostatic energy of the edge charges for the optimal choice of the
edge domain wall width $L_\eps = O(1)$ turns out to be the same as the
energy of the monodomain state. Therefore, for $h_\eps = O(1)$ it is
not energetically advantageous to form edge domain walls. These walls
would thus form at lower values of the exchange bias field $h_\eps$.

In order to balance the energies of the two configurations above for
$\delta_\eps$ given by \eqref{lamdeleps} and $h_\eps \ll 1$, we need
to evaluate the difference between the two at optimal wall width
$L_\eps = O(h_\eps^{-1/2})$. Matching the wall energy
$O(h_\eps^{1/2})$ with the energy difference
$O(\delta_\eps \ln (L_\eps / \delta_\eps))$ then yields that one needs
to choose
\begin{align}
  \label{hepslnln}
  h_\eps = \beta\brackets{\frac{\ln |\ln \eps|}{|\ln
    \eps|}}^2 ,
\end{align}
for some $\beta > 0$ playing the role of the renormalized field
strength. The corresponding optimal choice of $L_\eps$ is
$L_\eps = O(|\ln \eps| / (\ln |\ln \eps|))$. Furthermore, under
\eqref{lamdeleps} and \eqref{hepslnln} one would expect that a
transition from the monodomain state to states containing edge domain
walls takes place at some critical value of $\beta$ for fixed value of
$\lambda$ as $\eps \to 0$. Below we will show that this is indeed the
case and identify the critical value of $\beta$.

\section{Statement of results}
\label{sec:main-results}

We now proceed to formulate the main results of this paper. We begin
with the simplest setting, namely that of a periodic magnetization on
a strip oriented normally to the direction of the bias field as
described in Sec. \ref{sec:periodic}. Our main result here is the
identification of one-dimensional edge domain wall profiles as unique
global energy minimizers of the energy $E^\#$ irrespectively of the
relationship between $a$, $b$, $\delta$ and $h$. Throughout the rest
of this paper we always assume that $\delta < b/2$. 

We start by defining the admissible class in which we will seek the
minimizers of $E^\#$ 
\begin{align}
  \label{Aper}
  \mathcal A^\# := \{ m \in H^1_{loc}(\R \times [0,b]; \mathbb S^1): 
  m(x_1 + a, x_2) = m(x_1, x_2) \},
\end{align}
and introduce the representation of the magnetization in $\mathcal
A^\#$ in terms of the angle that $m$ makes with respect to the
$x_2$-axis:
\begin{align}
  \label{mthcossin}
  m = (-\sin \theta, \cos \theta).
\end{align}
We also define, for $\alpha \in (0,1)$, the one-dimensional
half-Laplacian acting on $u \in C^{1,\alpha}([0,b])$ that vanishes at
the endpoints, extended by zero to the rest of $\R$:
\begin{align}
  \label{halfL1d}
  \left( - {d^2 \over dx^2} \right)^{1/2} u(x) := {1 \over \pi}
  \dashint_0^b {u(x) - u(y) \over (x - y)^2} \d y + {b u(x) \over \pi x
  (b - x)} \qquad x \in (0,b).
\end{align}
Finally, with a slight abuse of notation we will use $\eta_\delta(x)$
to define the cutoff as a function of one variable, $x = x_2$ and
extend it by zero outside $(0,b)$.

We have the following basic characterization of the minimizers of
$E^\#$ over $\mathcal A^\#$. 

\begin{theorem}
  \label{t:1d}
  There exist at most three minimizers $m$ of $E^\#$ over
  $\mathcal A^\#$. Each minimizer is one-dimensional, i.e.,
  $m = m(x_2)$, and symmetric with respect to the midline, i.e.,
  $m(x_2) = m(b - x_2)$. Furthermore, $m_2(x_2) \geq 0$ and $m_1(x_2)$
  is either identically zero or does not change sign. In addition, if
  $\theta$ is such that $m$ satisfies \eqref{mthcossin}, then
  $\theta \in C^\infty(0,b) \cap C^2([0,b])$ and satisfies
  \begin{align}
    \label{EL1d}
   0 =  {d^2 \theta \over dx^2} - h \sin \theta + {\delta \over
    2} \eta_\delta \sin \theta \left( - {d^2 \over dx^2}
    \right)^{1/2} \eta_\delta \cos \theta \qquad x \in (0,b),
  \end{align}
  together with $\theta'(0) = \theta'(b) = 0$.
\end{theorem}

It is clear that $m = e_2$ is one possibility for a minimizer in
Theorem \ref{t:1d}, which corresponds to the monodomain state. Note
that by \eqref{EL1d} the state $m = e_2$ is always a critical point of
the energy $E^\#$. Furthermore, it is easy to see that $m = e_2$ is a
local minimizer of $E^\#$ if the Schr\"odinger-type operator
\begin{align}
  L = -{d^2 \over dx^2} + V(x), \qquad V(x) := h - {\delta \over 2}
  \eta_\delta(x)  \left( - {d^2 \over dx^2} \right)^{1/2} \eta_\delta(x)  
\end{align}
has only positive eigenvalues when $x \in (0, b)$. The monodomain
state competes with a profile having
$\theta = \theta(x_2) \in (0, \frac12 \pi]$ and another, symmetric
profile obtained by replacing $\theta$ with $-\theta$, both
corresponding to the edge domain walls.

\begin{rmk}
  Observe that by Theorem \ref{t:1d} the minimizers of $E^\#$ do not
  exhibit winding, i.e., the size of the range of $\theta$ associated
  with the minimizer does not reach or exceed $2 \pi$. Notice that a
  priori winding cannot be excluded, since the nonlocal term in the
  energy may favor oscillations of $m$. In fact, winding will be
  required if the minimization of $E^\#$ is carried out over an
  admissible class with a prescribed non-zero winding number across
  the period along $x_1$.
\end{rmk}

We now turn to the regime described in
Sec. \ref{sec:thin-film-regime}, in which edge domain walls emerge as
minimizers of $E^\#$. We begin by introducing a periodic version of
the rescaled energy in \eqref{Rescaled}:
\begin{align}
  \label{Esheps}
  E_\eps^\#(m) :=
  \frac{1}{2} \int_{D} \brackets{\eps \abs{\grad {m}}^2 + {h_\eps
  \over \eps} \abs{m-e_2}^2
  }\d x + \frac{\delta_\eps}{8\pi} \int_{D}\int_{\R \times (0,b)}
  \frac{\div m_{\eps \delta_\eps}(x) \div m_{\eps \delta_\eps}
  (y)}{\abs{x-y}}\d y \d x. 
\end{align}
This energy is still well defined on the admissible class
$\mathcal A^\#$ for $D = (0, a) \times (0, b)$. We are going to
completely characterize the minimizers of $E_\eps^\#$ under the
scaling assumptions in \eqref{lamdeleps} and \eqref{hepslnln} as
$\eps \to 0$. In particular, we will show that for small enough
$\beta$ the minimizers asymptotically consist of edge domain walls of
width of order $\eps L_\eps$, where
\begin{align}
  \label{Leps}
  L_\eps := {|\ln \eps| \over \ln |\ln \eps|}.
\end{align}
To see this, let us drop the nonlocal term in \eqref{Esheps} for the
moment and consider a magnetization profile $m$ given by
\eqref{mthcossin} with $\theta = \theta(x_2)$ satisfying
$\theta(0) = \theta_0 \in (0, \frac{\pi}{2}]$. Then after the
rescaling of $x_2$ by $\eps L_\eps$ and formally passing to the limit
$\eps \to 0$ we obtain the following local one-dimensional energy
\begin{align}
  E^\infty_{1d}(\theta) := \int_0^\infty \left(
  \frac12|\theta'|^2 + \beta (1 - \cos \theta) \right) \d x.
\end{align}
For $\theta_0$ fixed, this energy is explicitly minimized by
\begin{align}
  \label{thinf}
  \theta_\infty(x) = 4 \arctan \left( e^{2 \sqrt{\beta}(x_0 - x)} \right),
  \qquad x_0 = {1 \over 2 \sqrt{\beta}} \ln \tan \left( {\theta_0 \over
  4} \right), 
\end{align}
and the corresponding minimal energy is given by
\begin{align}
  \label{Einfthinf}
  E^\infty_{1d} (\theta_\infty) = 8 \sqrt{\beta} \sin^2 \left(
  {\theta_0 \over 4} \right). 
\end{align}
Indeed, using the Modica-Mortola trick \cite{modica87} we find that
\begin{align}
  \label{E1dthMM}
  E^\infty_{1d} (\theta) \geq -2 \sqrt{\beta} \int_0^\infty \sin \left( {\theta
  \over 2} \right) \theta' \d x + \frac12  \int_0^\infty \left[
  \theta' + 2 \sqrt{\beta} \sin \left( {\theta \over 2} \right)
  \right]^2 \d x \geq E^\infty_{1d}(\theta_\infty), 
\end{align}
and equality holds if and only if $\theta = \theta_\infty$.

We now define the function
\begin{align}
  \label{F0}
  F_0(n) := 4 \sqrt{\beta} \left( 1 - \sqrt{1 + n \over 2} \right) +
  {\lambda \over 4 \pi} ( 2 n^2 - 1 ) \qquad n \in [0,1],
\end{align}
and observe that $F_0(\cos \theta_0) = E^\infty_{1d}(\theta_\infty)$
when $\lambda = 0$. In the following we will show that, up to an
additive constant, the minimum of $E^\#_\eps$ may be bounded below as
$\eps \to 0$ by a multiple of $F_0(n_\eps)$, where $n_\eps$ is the
trace of the second component of the minimizer on the edge. Moreover,
this lower bound turns out to be sharp in the limit, allowing to
characterize the global energy minimizers of $E^\#_\eps$ in terms of
those of $F_0$. The latter can in principle be computed as roots of a
cubic polynomial, resulting in a cumbersome explicit formula. Taking
advantage of the fact that $F_0(n)$ is a strictly convex function of
$n$, however, one can conclude that $F_0$ admits a unique minimizer
for every $\lambda > 0$ and $\beta > 0$. We have the following result
regarding the minimizers of $F_0$, whose proof is a simple calculus
exercise.

\begin{lemma}
  \label{l:F0}
  Let $F_0(n)$ be defined by \eqref{F0} and let
  $n_0 = n_0(\beta, \lambda)$ be a minimizer of $F_0$ on $[0,1]$. Then
  $n_0$ is unique, and if
  \begin{align}
    \label{betac}
    \beta_c := {\lambda^2 \over \pi^2},
  \end{align}
  we have $n_0 = 1$ and $F_0(n_0) = {\lambda \over 4 \pi}$ for all
  $\beta \geq \beta_c$, while $0 < n_0 < 1$ and
  $F_0(n_0) < {\lambda \over 4 \pi}$ for all $\beta < \beta_c$.
\end{lemma}

\noindent We also remark that the bifurcation at $\beta = \beta_c$ can
be seen to be transcritical, and that $n_0(\beta, \lambda)$ is
monotone increasing in $\beta$ and goes to zero as $\beta \to 0$ with
$\lambda > 0$ fixed. The latter is consistent with the fact that the
magnetization wants to align tangentially to the film edge when the
energy at the edge is dominated by the stray field (see also
\cite{hornreich63,hornreich64,wade64, nonaka85, hirono86, 
  desimone00,kohn05arma,lms:non18}).

Our next result gives an asymptotic relation between the energy of the
minimizers of $E^\#_\eps$ and that of the minimizers of $F_0$.

\begin{theorem}
  \label{t:E1deps}
  Let $\lambda > 0$ and $\beta > 0$. Assume $\delta_\eps$ and $h_\eps$
  are given by \eqref{lamdeleps} and \eqref{hepslnln}. Then as
  $\eps \to 0$ we have
  \begin{align}
    \label{minEeps1d}
    {|\ln \eps| \over \ln |\ln \eps|} \left( \min_{m \in \mathcal
    A^\#} E_\eps^\#(m) - {a \lambda \over 2 \pi} \right) \to 2 a \min_{n
    \in [0,1]} F_0(n).
  \end{align}
\end{theorem}

We note that since $\min_{m \in \mathcal A^\#} E_\eps^\#$ is bounded
in the limit as $\eps \to 0$ and since the energy in \eqref{Esheps}
consists of the sum of three positive terms, we also get that
$m_\eps \to e_2$ in $L^2(D; \R^2)$ for any minimizer $m_\eps$ of
$E_\eps^\#$ (or even for any configuration with finite energy).
However, much more can be said about the minimizers of $E^\#_\eps$ in
the limit $\eps \to 0$, which is the content of our next theorem. Let
$m_\eps = (m_{\eps,1}, m_{\eps,2})$ be a minimizer, which by Theorem
\ref{t:1d} is one-dimensional, and define
\begin{align}
  \label{theps}
  \theta_\eps(x) := - \arcsin m_{\eps,1}(0, \eps L_\eps x) \qquad x
  \in (0, \eps^{-1} L_\eps^{-1} b),
\end{align}
where $L_\eps$ is defined in \eqref{Leps}. Then we have the following
result. 

\begin{theorem}
  \label{t:1dmin}
  Let $\lambda > 0$ and $\beta > 0$. Assume $\delta_\eps$ and $h_\eps$
  are given by \eqref{lamdeleps} and \eqref{hepslnln}, let $m_\eps$ be
  a minimizer of $E^\#$ over $\mathcal A^\#$, and let $\theta_\eps$ be
  defined in \eqref{theps}. Then as $\eps \to 0$ we have
  \begin{align}
    \label{convthinf}
    |\theta_\eps| \to \theta_\infty \ \text{in} \ H^1_{loc}(\overline{\R^+}),
  \end{align}
  where $\theta_\infty$ is given by \eqref{thinf} with
  $\theta_0 = \arccos n_0$ and $n_0$ is as in Lemma \ref{l:F0}.  In
  particular, $m_{2,\eps}(\cdot, 0) \to n_0$. Moreover, convergence in
  \eqref{convthinf} is uniform on
  $[0, \frac12 \eps^{-1} L_\eps^{-1} b ]$.
\end{theorem}

We remark that in view of the reflection symmetry of the minimizers
guaranteed by Theorem \ref{t:1d}, the same conclusions hold in the
vicinity of the top edge as well. We also note that by Theorem
\ref{t:1dmin} and Lemma \ref{l:F0}, there is a bifurcation from the
monodomain state to a state containing edge domain walls as the energy
minimizers at $\beta = \beta_c$ in the limit as $\eps \to 0$, with
$\theta_\infty = 0$ for all $\beta \geq \beta_c$ and
$\theta_\infty \not= 0$ for all $\beta < \beta_c$.

We now go to the original problem on the rectangular domain described
by the energy in \eqref{Rescaled}. In our final theorem, we establish
that both the energy of the minimizers and their average trace on the
top and the bottom edges of the rectangle approach the same values as
in the case of the minimizers in the periodic setting as $\eps \to 0$.

\begin{theorem}
  \label{t:rect}
  Let $\lambda > 0$ and $\beta > 0$. Assume $\delta_\eps$ and $h_\eps$
  are given by \eqref{lamdeleps} and \eqref{hepslnln}, and let
  $m_\eps$ be a minimizer of $E_\eps$ from \eqref{Rescaled} over
  $H^1(D; \mathbb S^1)$. Then as $\eps \to 0$ we have
  \begin{align}
    \label{minEeps1d2}
    {|\ln \eps| \over \ln |\ln \eps|} \left( E_\eps(m_\eps) - {a
    \lambda \over 2 \pi} \right) \to 2 a F_0(n_0),
  \end{align}
  where $n_0 \in [0,1]$ is the unique minimizer of $F_0$ in
  \eqref{F0}. Furthermore, $m_\eps(x) \to e_2$ for a.e. $x \in D$, and
  we have
  \begin{align}
    \label{eq:nepsn0}
    {1 \over a} \int_0^a m_{2,\eps}(t, 0) \d t \to n_0 \qquad
    \text{and} \qquad  {1 \over a} \int_0^a m_{2,\eps}(t, b) \d t
    \to n_0.  
  \end{align}
\end{theorem}

The statement of the above theorem implies that when $D$ is a
rectangle aligned with the direction of the preferred magnetization
the minimal energy behaves asymptotically as twice the horizontal edge
length times the energy of the one-dimensional edge domain wall, while
the average trace of the minimizer at the top and bottom edges agrees
with that in the one-dimensional edge domain wall. At the same time,
the magnetization in the bulk tends to its preferred value $m =
e_2$. This is consistent with the expectation that a one-dimensional
boundary layer should form near the charged edges.

\section{Proof of Theorem \ref{t:1d}}
\label{sec:proof-theor-reft:1d}

First of all, existence of a minimizer $m \in \mathcal A^\#$ follows
from the direct method of calculus of variations, using standard
arguments. To prove that the minimizer is one-dimensional, for any
admissible $m$ we define a competitor
$\overline m = (\overline m_1, \overline m_2)$, where
\begin{align}
  \label{mbar12}
  \overline m_2(x_1, x_2) := {1 \over a} \int_0^a m_2(t, x_2) \d t,
  \qquad \overline m_1(x_1, x_2) := \sqrt{1 - \overline m_2^2(x_1,
  x_2)}. 
\end{align}
We are now going to establish several useful results concerning
$\overline m$.

\begin{lemma}
  \label{l:madmiss}
  Let $m \in \mathcal A^\#$ and let $\overline m$ be defined by
  \eqref{mbar12}. Then $\overline m \in \mathcal A^\#$,
  \begin{align}
    \label{gradmave}
    \int_D |\nabla \overline m|^2 \d x \leq  \int_D |\nabla m|^2 \d x,
  \end{align}
  and equality in the above expression holds if and only if $m$ is
  independent of $x_1$.
\end{lemma}

\begin{proof}
  Since $m(x) = (m_1(x), m_2(x)) \in \mathbb S^1$ for a.e. $x \in D$,
  we have
  \begin{align}
    \label{m1m2unit}
    m_1^2(x) + m_2^2(x) = 1 \qquad \text{for a.e.} \ x \in D.
  \end{align}
  Therefore, applying weak chain rule \cite[Theorem 6.16]{lieb-loss}
  to the above expression yields
  \begin{align}
    \label{m1m2grad}
    m_1 \nabla m_1 = - m_2 \nabla m_2  \qquad \text{a.e. in} \ D. 
  \end{align}
  Combining \eqref{m1m2unit} and \eqref{m1m2grad}, and using the fact
  that $\nabla m_1(x) = 0$ for a.e. $x \in A \subseteq D$ whenever
  $m_1 = 0$ on $A$ and $|A| > 0$ \cite[Theorem 6.19]{lieb-loss}, we
  have
  \begin{align}
    \label{gradmm2}
    |\nabla m(x)| =
    \begin{cases}
      {|\nabla m_2(x)| \over \sqrt{1 - m_2^2}}, & |m_2(x)| < 1, \\
      0, & |m_2(x)| = 1,
    \end{cases}
    \qquad \text{for a.e.} \ x \in D.
  \end{align}
  Note that this implies $\nabla m_2 = 0$ on the set $A$ as well.
  Then by monotone convergence theorem we can write
  \begin{align}
    \label{gradm2eps}
    \int_D |\nabla m|^2 \d x = \lim_{\eps \to 0} \int_D {|\nabla m_2|^2
    \over 1 + \eps - m_2^2} \d x.
  \end{align}
  Now for $\eps > 0$ consider the function
  \begin{align}
    \label{Fuveps}
    F_\eps(u, v) := {v^2 \over 1 + \eps - u^2} \qquad (u, v) \in
    [-1,1] \times \R.
  \end{align}
  By direct computation this function is convex for all $\eps >
  0$. Therefore
  \begin{align}
    \int_D {|\nabla m_2|^2 \over 1 + \eps - m_2^2} \d x
    & = \int_D {|\partial_1
      m_2|^2 \over 1 + \eps - m_2^2} \d x + \int_D
      F_\eps(m_2, \partial_2 m_2) \d x \notag \\
    & \geq \int_D
      F_\eps(\overline m_2, \partial_2 \overline m_2) \d x
      + \int_D \partial_u F_\eps(\overline m_2, \partial_2 \overline
      m_2) (m_2 - \overline m_2) \d x \notag \\
    & \qquad \qquad + \int_D \partial_v F_\eps(\overline m_2,
      \partial_2 \overline m_2) (\partial_2 m_2 - \partial_2 \overline
      m_2) \d x.  \label{Fepsuv} 
  \end{align}
  At the same time, by Fubini's theorem and the definition of
  $\overline m_2$ we have
  \begin{align}
    \int_D \partial_u F_\eps(\overline m_2, \partial_2 \overline
    m_2) (m_2 - \overline m_2) \d x
    = \int_0^b \left( \partial_u F_\eps(\overline
    m_2, \partial_2 \overline m_2) \int_0^a (m_2- \overline m_2) \d
    x_1 \right) \d x_2 = 0,
  \end{align}
  and
  \begin{align}
    \int_D \partial_v F_\eps(\overline m_2, \partial_2 \overline
    m_2) (\partial_2 m_2 - \partial_2 \overline m_2) \d x
    = \int_0^b
    \left( \partial_v F_\eps(\overline m_2, \partial_2 \overline m_2)
    \int_0^a (\partial_2  m_2- \partial_2 \overline m_2) \d x_1 \right)
    \d x_2 \notag \\
    = \int_0^b
    \left( \partial_v F_\eps(\overline m_2, \partial_2 \overline m_2)
    \, \partial_2 \int_0^a (m_2- \overline m_2) \d x_1 \right)
    \d x_2= 0.
  \end{align}
  This yields
  \begin{align}
    \label{gradm2m2bareps}
    \int_D {|\nabla m_2|^2 \over 1 + \eps - m_2^2} \d x \geq \int_D
    {|\nabla \overline m_2|^2 \over 1 + \eps - \overline m_2^2} \d x.
  \end{align}

  We now argue by approximation and take
  $m^\delta \in C^\infty(\R \times [0,b]; \mathbb S^1)$ such that
  $m^\delta \to m$ in $H^1_{loc}(\R \times [0,b]; \R^2)$ as
  $\delta \to 0$ \cite{bethuel88,bourgain00}. Then we have
  $\overline m_2^\delta \in C^\infty(\R \times [0,b])$ as
  well. Turning to $\overline m_1^\delta$ defined in \eqref{mbar12},
  observe that $\overline m_1^\delta \in C(\R \times
  [0,b])$. Furthermore, since $\overline m_1$ is a composition of a
  smooth non-negative function with the square root, we also have that
  $\overline m_1^\delta \in W^{1,\infty}(\R \times (0,b))$. Thus,
  $\overline m^\delta \in H^1_{loc}(\R \times [0,b])$, and by the
  arguments at the beginning of the proof we have
  \begin{align}
    \label{gradmoverl2eps}
    \int_D |\nabla \overline m^\delta|^2 \d x = \lim_{\eps \to 0} \int_D
    {|\nabla \overline m_2^\delta|^2
    \over 1 + \eps - |\overline m_2^\delta|^2} \d x.
  \end{align}
  Combining this equality with \eqref{gradm2eps} and
  \eqref{gradm2m2bareps}, we arrive at \eqref{gradmave} for $m^\delta$
  and $\overline m^\delta$. Passing to the limit $\delta \to 0$, by
  lower semicontinuity of $\int_D |\nabla \overline m^\delta|^2 \d x$
  we obtain that $\overline m_1 \in H^1_{loc}(\R \times [0,b])$ and
  \eqref{gradmave} holds. Furthermore, by construction
  $|\overline m| = 1$, and $\overline m$ is independent of $x_1$,
  hence, $\overline m \in \mathcal A^\#$. Finally, if equality holds
  in \eqref{gradmave} then we have
  $\int_D |\partial_1 m_2|^2 \d x = 0$, yielding the rest of the
  claim.
\end{proof}

With a slight abuse of notation, from now we will frequently refer to
$\overline m$ as a function of one variable, i.e.,
$\overline m = \overline m(x_2)$, and extend it by zero for all
$x_2 \not \in (0,b)$. Similarly, we treat $\eta_\delta$ in
\eqref{mdelta} as a function of one variable, i.e.,
$\eta_\delta = \eta_\delta(x_2)$, and extended it by zero for all
$x_2 \not \in (0,b)$ as well.

\begin{lemma}
  \label{l:nonloc1d}
  Let $m \in \mathcal A^\#$. Then
  \begin{align}
    \label{eq:nonloc1d}
    \int_D \int_{\R \times (0,b)} {\nabla \cdot m_\delta (x) \nabla
    \cdot m_\delta (y) \over |x - y|} \d x \d y \geq a
    \int_\R \int_\R {(\overline m_2(x) \eta_\delta(x) - \overline
    m_2(y) \eta_\delta(y))^2 \over (x - y)^2} \d x \d y,
  \end{align}
  where $m_\delta$ is defined in \eqref{mdelta0} and $\overline m$ is
  given by \eqref{mbar12}. Moreover, equality holds if and only if
  $m_2(x) = \overline m_2(x)$ for a.e. $x \in D$.
\end{lemma}

\begin{proof}
  The proof proceeds via passing to Fourier space. For
  $n \in \mathbb Z$ and $\xi \in \R$, we define Fourier coefficients
  $c(n, \xi) \in \R^2$ as
  \begin{align}
    c(n, \xi) := \int_D e^{- i q(n,\xi) \cdot x} m_\delta(x) \d x,
  \end{align}
  where $q(n, \xi) := (2 \pi a^{-1} n, \xi) \in \R^2$.  Then the inversion
  formula reads (see, e.g., \cite[Section 4]{milisic13}):
  \begin{align}
    \label{mdelinv}
    m_\delta(x) = {1 \over 2 \pi a} \sum_{n \in \mathbb Z} \int_\R e^{i
    q(n,\xi) \cdot x} c(n, \xi) \d \xi. 
  \end{align}
  In terms of $c(n, \xi)$ the left-hand side of \eqref{eq:nonloc1d} may
  be written as
  \begin{align}
    \label{nloc1dfour}
    \int_D \int_{\R \times (0,b)} {\nabla \cdot m_\delta (x) \nabla
    \cdot m_\delta (y) \over |x - y|} \d x \d y = {1 \over a} \sum_{n
    \in \mathbb Z} \int_\R {|q(n, \xi) \cdot c(n, \xi)|^2 \over
    |q(n, \xi)|} \d \xi.
  \end{align}
  Keeping only the $n = 0$ contribution in the right-hand side, we,
  therefore, have
  \begin{align}
    \int_D \int_{\R \times (0,b)} {\nabla \cdot m_\delta (x) \nabla
    \cdot m_\delta (y) \over |x - y|} \d x \d y \geq {1 \over a}
    \int_\R |\xi| \, |c_2(0, \xi)|^2 \d \xi. 
  \end{align}
  Passing back to real space, with the help of the integral formula
  for the $\mathring H^{1/2}(\R)$ norm \cite{dinezza12} we obtain
  \eqref{eq:nonloc1d}. Finally, by \eqref{mdelinv} and
  \eqref{nloc1dfour} the inequality in \eqref{eq:nonloc1d} is strict,
  unless $m_2 = \overline m_2$ almost everywhere.
\end{proof}

Having obtained the above auxiliary results for $\overline m$, we now
proceed to the proof of our first theorem.

\begin{proof}[Proof of Theorem \ref{t:1d}]
  Let $m \in \mathcal A^\#$ be a minimizer of $E^\#$. By Lemmas
  \ref{l:madmiss} and \ref{l:nonloc1d}, we have
  $E^\#(m) \geq E^\#(\overline m)$, where $\overline m$ is defined in
  \eqref{mbar12}. In particular, this inequality is in fact an
  equality, and by Lemma \ref{l:madmiss} we have $m =
  m(x_2)$. Moreover, by Lemma \ref{l:nonloc1d} we have
  $E^\#(m) = a E^\#_{1d}(m)$, where
  \begin{align}
    \label{Esh1d}
    E^\#_{1d}(m) := \frac12 \int_0^b \left( |m'|^2 + h |m -
    e_2|^2 \right) \d x + {\delta \over 8 \pi} \int_\R \int_\R {
    (m_2(x) \eta_\delta(x) - m_2(y) \eta_\delta(y))^2  \over (x -
    y)^2} \d x \d y,
  \end{align}
  with the usual abuse of notation that $m$ and $\eta_\delta$ are
  treated as functions of one variable in the right-hand side of
  \eqref{Esh1d}, and $m_2 \eta_\delta$ has been extended by zero
  outside $(0,b)$.

  We now claim that $m_2(x_2) \geq 0$ for all $x_2 \in (0,b)$. Indeed,
  taking $\widetilde m := (m_1, |m_2|) \in \mathcal A^\#$ as a
  competitor, we have $|\nabla \widetilde m| = |\nabla m|$ and
  \begin{align}
    \label{H12tilm}
    \int_\R \int_\R {
    (m_2(x) \eta_\delta(x) - m_2(y) \eta_\delta(y))^2  \over (x -
    y)^2} \d x \d y \geq   \int_\R \int_\R {
    (\widetilde m_2(x) \eta_\delta(x) - \widetilde m_2(y)
    \eta_\delta(y))^2 \over (x - y)^2} \d x \d y,
  \end{align}
  where the last inequality follows from the fact that the integrand
  in the right-hand side of \eqref{H12tilm} is pointwise no greater
  than that in its left-hand side. On the other hand, since
  $|m - e_2|^2 = 2 - 2 m_2$, we have
  $E^\#_{1d}(m) > E^\#_{1d}(\widetilde m)$, unless
  $\widetilde m(x_2) = m(x_2)$ for all $x_2 \in (0,b)$.

  Now that we established that $m_2 \geq 0$, we may define
  $\theta(x_2) := -\arcsin m_1(x_2) \in [-\frac{\pi}{2},
  \frac{\pi}{2}]$, so that $m$ satisfies \eqref{mthcossin}. Then we
  can rewrite the energy of the minimizer as
  \begin{align}
    \label{Eshth}
    E^\#_{1d}(m) = \int_0^b \left( \frac12 |\theta'|^2 +  h (1
    - \cos \theta) \right) \d x + {\delta \over 8 \pi} \int_\R \int_\R {
    (\cos \theta(x) \eta_\delta(x) - \cos \theta(y) \eta_\delta(y))^2  \over (x -
    y)^2} \d x \d y,
  \end{align}
  where in the exchange energy we approximated $\theta$ by functions
  bounded away from $\pm \frac{\pi}{2}$ and passed to the limit with
  the help of monotone convergence theorem. In particular, from
  boundedness of the right-hand side of \eqref{Eshth} it follows that
  $\theta \in H^1(0,b)$. Therefore, $\theta$ satisfies the weak form
  of \eqref{EL1d} (for further details, see
  \cite{capella07,cm:non13}). At the same time, since
  $\eta_\delta \cos \theta \in H^1(\R)$ by weak product and chain
  rules \cite[Corollaries 8.10 and 8.11]{brezis}, and the operator
  $(-d^2 / dx^2)^{1/2}$ is a bounded linear operator from $H^1(\R)$ to
  $L^2(\R)$, we also have $\theta'' \in L^2(0,b)$, and, hence,
  $\theta \in C^{1,1/2}([0,b])$. In particular, we can use the formula
  in \eqref{halfL1d} to compute the non-local term in \eqref{EL1d}.

  We now apply a bootstrap argument to establish further interior
  regularity of $\theta$. Note that this result is not immediate,
  since the function $\eta_\delta$ extended by zero to the whole real
  line is only Lipschitz continuous. Nevertheless, for every $x \in I$
  where $I \Subset (0,b)$ is open we can introduce a partition of
  unity whereby we have
  \begin{align}
    \label{partuni}
    \left(-  {d^2 \over dx^2} \right)^{1/2} \eta_\delta(x) \cos \theta(x)
    = {1
    \over \pi} \dashint_\R {\eta_\delta(x) \cos \theta(x) - \eta_\delta(y)
    \cos \theta(y) \chi(y) \over (x - y)^2} \d y \notag \\
    \qquad -  {1
    \over \pi} \int_\R { \eta_\delta(y)
    \cos \theta(y) (1 - \chi(y)) \over (x - y)^2} \d y,
  \end{align}
  where $\chi \in C^\infty_c(\R)$ is such that $\chi \equiv 1$ in $I$ and
  $\text{supp}(\chi) \subset (0,b)$. Taking the distributional
  derivative of the right-hand side in \eqref{partuni} and using the
  fact that now $\eta_\delta \chi \cos \theta \in H^2(\R)$, we get
  that the left-hand side of \eqref{partuni} is in $H^1(I)$. Applying
  the bootstrap argument locally, we thus obtain that $\theta \in
  H^3_{loc}(0,b)$ and, hence, $\theta \in C^\infty(0,b)$, and \eqref{EL1d}
  holds classically for all $x \in (0,b)$. Once the latter is
  established, we obtain the boundary condition $\theta'(0) =
  \theta'(b) = 0$ via integration by parts.

  To establish higher regularity of $\theta$ near the boundary, we
  estimate the nonlocal term, using the fact that
  $\eta_\delta \in C^\infty([0,b])$ and $\theta' \in
  C^{1/2}([0,b])$. For $x \in (0,b)$ let
  $u(x) := \eta_\delta(x) \cos \theta(x)$. Notice that
  \begin{align}
    \label{uCxbx}
    |u(x)| \leq C x (b - x),
  \end{align}
  for some $C > 0$. Focusing on the first term in the right-hand side
  of \eqref{halfL1d}, with the help of Taylor formula we can write for
  $x \in (0, \frac12 b)$:
  \begin{align}
    \left| \dashint_0^b {u(x) - u(y) \over (x - y)^2} \d y \right|
    \leq \left| \dashint_0^{2x} {u(x) - u(y) \over (x - y)^2} \d y
    \right| + \left| \int_{2x}^b {u(x) - u(y) \over (x - y)^2} \d y
    \right| \notag \\
    \leq \int_0^{2x} {|u'(\xi_1) - u'(x)| \over |x - y|} \d y +
    \int_{2x}^b {|u'(\xi_2)| \over |x - y|} \d y \notag \\
    \leq C x^{1/2} + C \ln (2 b/x),
  \end{align}
  for some $C > 0$, where $|\xi_1 - x| < |x - y|$ and $\xi_2 \in (x,
  y)$. Combining this with \eqref{uCxbx} yields
  \begin{align}
    \label{etadhaflap}
    \left| \eta_\delta(x) \left( -{ d^2 \over dx^2} \right)^{1/2}
    \eta_\delta(x) \cos \theta(x) \right| \leq C x (1 + x^{1/2} + \ln
    x^{-1} ),
  \end{align}
  for some $C > 0$ and all $x$ sufficiently small. Thus, the
  expression in the left-hand side of \eqref{etadhaflap} is continuous
  and vanishes at $x = 0$. By the same argument, the same holds true
  near $x = b$. Using this fact, from \eqref{EL1d} we conclude that
  $\theta \in C^2([0,b])$.

  We now prove that there are at most three minimizers of $E^\#$ in
  $\mathcal A^\#$. Let $m$ be a minimizer associated with
  $\theta \in H^1(0,b)$. Then by \eqref{Eshth} the function
  $\widetilde m \in \mathcal A^\#$ associated with
  $\tilde \theta = |\theta|$ is also a minimizer. In particular,
  $\tilde \theta \in C^2([0,b])$ and solves \eqref{EL1d}
  classically. Now, suppose that there exists a point $x_0 \in [0,b]$
  such that $\tilde \theta(x_0) = 0$. By regularity of $\tilde \theta$
  in the interior or homogeneous Neumann boundary conditions we then
  also have $\tilde \theta'(x_0) = 0$. We now apply a maximum
  principle type argument based on the uniqueness of the solution of
  the initial value problem for \eqref{EL1d} considered as an ordinary
  differential equation with the nonlocal term treated as a given
  function of $x \in [0,b]$:
  \begin{align}
    \theta''(x) = c(x) \sin \theta(x), \quad
    c(x) := h - {\delta \over 2} \eta_\delta(x) \left( -{ d^2 \over
    dx^2} \right)^{1/2} \eta_\delta(x) \cos \theta(x).
  \end{align}
  Indeed, by the argument in the preceding paragraph the function $c(x)$ is
  continuous on $[0,b]$. Therefore, if $\tilde \theta(x)$ vanishes for
  some $x_0 \in [0,b]$ we have $\tilde \theta \equiv 0$ on
  $[0,b]$. Alternatively, $\tilde \theta > 0$ for all $x \in [0,b]$,
  which means that $\theta$ does not change sign.

  To conclude the proof of the multiplicity of the minimizers, observe
  that in view of the above we need to show that there is at most one
  minimizer $\theta \in (0, {\pi \over 2}]$ of the right-hand side of
  \eqref{Eshth}. In this case we can rewrite the energy in terms of
  $m_2 < 1$:
  \begin{align}
    \label{E1dm2}
    E^\#_{1d}(m) = \frac12 \int_0^b \left( {|m_2'|^2 \over 1 - m_2^2}
    + 2 h ( 1 - m_2) \right) \d x + {\delta \over 8 \pi} \int_\R
    \int_\R {(m_2(x) \eta_\delta(x) - m_2(y) \eta_\delta(y))^2 \over
    (x - y)^2} \d x \d y.
  \end{align}
  By inspection this energy is convex. Furthermore, the last term in
  \eqref{E1dm2} is strictly convex in view of the fact that
  $m_2 \eta_\delta$ vanishes identically outside $(0,b)$. Thus, there
  is at most one minimizer with $m_1 > 0$. If such a minimizer exists,
  then by reflection symmetry the function
  $\widetilde m := (-m_1, m_2)$ is also a minimizer, which is the only
  minimizer with $\widetilde m_1 < 0$. Finally, the symmetry of the
  minimizer with respect to reflections $x_2 \to b - x_2$ follows from
  the invariance of the energy in \eqref{E1dm2} with respect to such
  reflections.
\end{proof}

\section{Proof of Theorems \ref{t:E1deps} and \ref{t:1dmin}}
\label{sec:proof-theor-23}

In view of the result in Theorem \ref{t:1d}, it suffices to consider
the minimizers of a suitably rescaled version of the one-dimensional
energy in \eqref{E1dm2} when $m_2 < 1$:
\begin{align}
  \label{E1dm2eps}
  E^\#_{\eps,1d}(m) := \frac12 \int_0^b \left( {\eps
  |m_2'|^2 \over 1 - m_2^2} + {2 h_\eps \over \eps} ( 1 - m_2) \right)
  \d x + {\delta_\eps \over 8 \pi} \int_\R \int_\R {(m_2(x) \eta_{\eps
  \delta_\eps}(x) - m_2(y) \eta_{\eps \delta_\eps}(y))^2 \over (x -
  y)^2} \d x \d y. 
\end{align}
Let us also define a rescaled version of this energy, up to an
additive constant:
\begin{align}
  \label{Feps}
  F_{\eps,1d}(m)
  & := \frac12 \int_0^{\eps^{-1} L_\eps^{-1} b} \left( {
    |m_2'|^2 \over 1 - m_2^2} + 2 \beta ( 1 - m_2) \right) \d x
    - {\lambda |\ln \eps| \over 2 \pi \ln |\ln \eps|} \notag \\
  & \qquad\qquad + {\lambda \over 8 \pi \ln |\ln \eps|} \int_\R
    \int_\R {(m_2(x) \widetilde \eta_{\delta_\eps / L_\eps}(x) -
    m_2(y) \widetilde \eta_{\delta_\eps / L_\eps}(y))^2 \over (x -
    y)^2} \d x \d y,  
\end{align}
where
$\widetilde \eta_{\delta_\eps / L_\eps}(x) := \eta(L_\eps \min(x,
\eps^{-1} L_\eps^{-1} b - x) / \delta_\eps)$. Using these definitions,
we have
\begin{align}
  \label{Eeps1dFeps}
  E^\#_{\eps,1d}(m) = {\lambda \over 2 \pi} + {\ln |\ln
  \eps| \over |\ln \eps|}  F_{\eps,1d}(m(\cdot / (\eps L_\eps)).
\end{align}
With these notations, proving Theorem \ref{t:E1deps} is equivalent to
showing that $\min F_{\eps,1d}(m(\cdot / (\eps L_\eps))$ converges to
$2 F_0(n_0)$ as $\eps \to 0$, where the minimization is done over
\begin{align}
  \label{Aeps1d}
  \mathcal A_\eps^{1d} := H^1((0, \eps^{-1} L_\eps^{-1} b); \mathbb S^1).
\end{align}
Below we show that this is indeed the case by establishing the
matching upper and lower bounds for $\min F_{\eps,1d}$.

To proceed, we separate the energy $F_{\eps,1d}$ into the local and
the non-local parts:
\begin{align}
  F_{\eps,1d}(m) = F_{\eps,1d}^{MM}(m) + F_{\eps,1d}^S(m),
\end{align}
where
\begin{align}
  \label{FepsMM}
  F_{\eps,1d}^{MM}(m) := \frac12 \int_0^{\eps^{-1} L_\eps^{-1} b} \left( {
    |m_2'|^2 \over 1 - m_2^2} + 2 \beta ( 1 - m_2) \right) \d x
\end{align}
is the Modica-Mortola type energy and
\begin{align}
  \label{FepsS}
  F_{\eps,1d}^S(m) :=  {\lambda \over 8 \pi \ln |\ln \eps|} \int_\R
  \int_\R {(m_2(x) 
  \widetilde \eta_{\delta_\eps / L_\eps}(x) - m_2(y) \widetilde
  \eta_{\delta_\eps / 
  L_\eps}(y))^2 \over (x - y)^2} \d x \d y - {\lambda |\ln \eps| \over
  2 \pi \ln |\ln \eps|} 
\end{align}
is the stray field energy, up to an additive constant. Note that using
the standard Modica-Mortola trick \cite{modica87}, one obtains a lower
bound for $F_{\eps,1d}^{MM}$.
\begin{lemma}
  \label{l:MM}
  Let $m = (m_1, m_2) \in \mathcal A_\eps^{1d}$ with $0 \leq m_2 <
  1$. Then for every $R \in (0, \eps^{-1} L_\eps^{-1} b/2]$ and every
  $r \in [0, R]$ we have
  \begin{align}
    \label{FepsMMlb}
    F_{\eps,1d}^{MM}(m) \geq 4 \sqrt{\beta} \left( 1 - \sqrt{1 + m_2(r)
    \over 2} \right) + 4 \sqrt{\beta} \left( 1 - \sqrt{1 +
    m_2(\eps^{-1} L_\eps^{-1} b - r)
    \over 2} \right) \notag \\
    - 8 \sqrt{\beta} \left( 1 - \sqrt{1 + m_2(R)
    \over 2} \right).
  \end{align}
\end{lemma}

\noindent In order to obtain the upper and lower bounds on the stray
field energy we prove the following lemma that offers two
characterizations of the one-dimensional fractional homogeneous
Sobolev norm. Here, by $\mathring{H}^1(\R^2)$ we understand the space
of functions in $L^2_{loc}(\R^2)$ whose distributional gradient is in
$L^2(\R^2; \R^2)$. 

\begin{lemma}\label{l:Str}
  Let $u \in H^1(\R)$ and have compact support. Then
  \begin{enumerate}[(i)]
  \item
    \begin{equation}
      \label{StrEq}
      \frac{1}{4 \pi} \int_\R \int_\R {(u(x) - u(y) )^2 \over (x -
        y)^2} \d x \d y =  {1 \over 2 \pi} \int_\R
      \int_\R \ln |x - y|^{-1} u'(x) u'(y) \d x \d y.
    \end{equation}
  \item 
  \end{enumerate}
  \begin{align}
    \label{StrIneq}
    \frac{1}{4\pi} \int_\R \int_\R {(u(x) - u(y) )^2 \over (x - y)^2}
    \d x \d y  =  -  \min_{v \in \mathring{H}^1(\R^2)} \left( \int_\R
    \int_\R |\nabla v(x,z)|^2 \d x \d z   +   2 \int_\R v(x,0) u'(x)
    \d x \right). 
  \end{align}
\end{lemma}

\begin{proof}
  For the proof of \eqref{StrEq}, we refer to the Appendix in
  \cite{lms:non18}. To obtain \eqref{StrIneq}, we first note that the
  minimum in the right-hand side of \eqref{StrIneq} is
  attained. Indeed, considering the elements of the homogeneous
  Sobolev space $\mathring{H}^1(\R^2)$ as equivalence classes of
  functions modulo additive constants makes this space into a Banach
  space \cite{ortner12}, and by coercivity and strict convexity of the
  expression in the brackets we hence get existence of a unique
  minimizer (up to an additive constant). Note that the integrals in
  the right-hand side of \eqref{StrIneq} are unchanged when an
  arbitrary constant is added to $v$, and that
  $v(\cdot, 0) \in L^2_{loc}(\R)$ is well defined as the trace of a
  Sobolev function.

  The minimizer $v_0 \in \mathring{H}^1(\R^2)$ of the expression in
  the right-hand side of \eqref{StrIneq} solves the following Poisson
  type equation
  \begin{align}
    \label{eq:v}
    \Delta v_0 = u'(x) \delta(z) \qquad \text{in} \ \mathcal D'(\R^2),
  \end{align}
  where $\delta(\cdot)$ is the one-dimensional Dirac
  delta-function. Therefore, $v_0$ is easily seen to be (again, up to
  an additive constant)
  \begin{align}
    \label{vharm}
    v_0(x, z) = {1 \over 2 \pi} \int_\R u'(y) \ln \sqrt{(x - y)^2 + z^2} \,
    \d y. 
  \end{align}
  In particular, since $u'$ has compact support and, therefore,
  integrates to zero over $\R$, we have an estimate for the function
  $v_0$ in \eqref{vharm}:
  \begin{align}
    \label{v0decay}
    |v_0(x,z)| \leq {C \over \sqrt{x^2 + z^2}} \qquad  |\nabla
    v_0(x,z)| \leq {C \over x^2 + z^2},
  \end{align}
  for some $C > 0$ and all $x^2 + z^2$ large enough. Furthermore, it
  is not difficult to see that $v_0 \in C^{1/2}(\R^2)$:
  \begin{align}
    |v(x_1, z_1) - v(x_2, z_2)|^2 \leq {1 \over 16 \pi^2} \int_\R
    |u'(y)|^2 \d y \int_\R \ln^2 \left\{ {(y - x_1)^2 + z_1^2 \over (y -
    x_2)^2 + z_2^2} \right\} \d y,
  \end{align}
  where we used Cauchy-Schwarz inequality, and the last integral may
  be dominated by $C (|x_1 - x_2| + |z_1 - z_2|)$ for some universal
  $C > 0$.

  We now multiply both parts of \eqref{eq:v} by $v_0$ and integrate
  over $\R^2$. After integrating by parts and taking into account
  \eqref{v0decay}, we obtain
  \begin{align}
    \label{vbyparts}
    \int_\R \int_\R |\nabla v_0 (x,z)|^2  \d x \d z = - \int_\R  v_0
    (x, 0) u'(x) \d x.
  \end{align}
  From this, we get
  \begin{align}
    \label{v0Str}
    \min_{v \in \mathring{H}^1(\R^2)} \left( \int_\R \int_\R
    |\nabla v(x,z)|^2 \d x \d z   +   2 \int_\R v(x,0) u'(x) \d x
    \right) = \int_\R  v_0 (x, 0) u'(x) \d x. 
  \end{align} 
  Finally, combining \eqref{StrEq}, \eqref{vharm} and \eqref{v0Str},
  we obtain \eqref{StrIneq}.
\end{proof}

Using the definition of $F_{\eps,1d}^S(m)$ and Lemma~\ref{l:Str}, we
arrive at the following lower bound for the stray field energy.
\begin{lemma}
  \label{l:FepsSlb}
  Let $m \in \mathcal A_\eps^{1d}$. Then
  \begin{align}
    \label{FepsSlb}
    F^S_{\eps,1d}(m) \geq     - {\lambda |\ln \eps| \over 2 \pi \ln
    |\ln \eps|} - {\lambda \over 2 \ln |\ln \eps|}
    \int_\R \int_\R |\nabla v(x,z)|^2 \d x \d z \notag \\
    -  {\lambda \over \ln |\ln \eps|} \int_0^{\eps^{-1} L_\eps^{-1} b}
    v(x,0)  \left( m_2(x) \widetilde \eta_{\delta_\eps /
    L_\eps}(x) \right)' \d x, 
  \end{align}
  for every $v \in \mathring{H}^1(\R^2)$, where $v(\cdot, 0)$ is
  understood in the sense of trace.
\end{lemma}

We will also find useful the following basic upper bound for the
minimum energy. 

\begin{lemma}
  \label{l:upe2}
  There exists $C > 0$ such that
  \begin{align}
    \label{Fube2}
    \min_{m \in \mathcal A_\eps^{1d}} F_{\eps,1d}(m) \leq C,
  \end{align}
  for all $\eps$ sufficiently small. Furthermore, if $F_{\eps,1d}$ is
  minimized by $m = e_2$, then the reverse inequality also holds. 
\end{lemma}

\begin{proof}
  The proof is obtained by testing the energy with $m = e_2$. Then
  $F_{\eps,1d}^{MM}(m) = 0$, and by Lemma \ref{l:Str} we have
  \begin{align}
    4 \pi \lambda^{-1} \ln |\ln \eps| F_{\eps,1d}^S(m) +
    2 |\ln \eps| 
    & = \int_0^{b
      \over \eps L_\eps}
      \int_0^{b \over \eps L_\eps} \ln |x - y|^{-1} \widetilde
      \eta_{\delta_\eps / 
      L_\eps}'(x) \widetilde \eta_{\delta_\eps / L_\eps}'(y) \d x \d y
      \notag \\ 
    & = 2 \int_0^{b \over \eps \delta_\eps} \int_0^{b \over \eps
      \delta_\eps} \ln 
      |x - y|^{-1} \eta'(x) \eta'(y) \d x \d y 
      \notag \\
    & \qquad - 2 \int_0^{b \over \eps \delta_\eps} \int_0^{b \over \eps
      \delta_\eps}  \ln |x - y|^{-1} \eta'(x) \eta'(\eps^{-1} 
      \delta_\eps^{-1} b - y) \d x \d y  \notag \\
    & \leq C + 2 \ln (\eps^{-1} \delta_\eps^{-1} b)\leq 2 |\ln \eps| +
      2 \ln |\ln \eps| + C', 
  \end{align}
  for some $C, C' > 0$ and all $\eps \ll 1$, where we took into
  account \eqref{lamdeleps}.  This inequality is equivalent to
  \eqref{Fube2}. Finally, if $m = e_2$ is the minimizer, the matching
  asymptotic lower bound then follows.
\end{proof}

\begin{proof}[Proof of Theorem \ref{t:E1deps}.]
  Let $m_\eps$ be a minimizer of $F_{\eps,1d}$ over
  $\mathcal A_\eps^{1d}$. Note that in view of Lemma \ref{l:upe2} and
  Theorem \ref{t:1d} we may assume that $m_{2,\eps} < 1$. With the
  help of the rescalings introduced earlier, proving Theorem
  \ref{t:E1deps} amounts to establishing that
  \begin{align}
    \label{F0infsup1d}
    2 F_0(n_0) \leq \liminf_{\eps \to 0} F_{\eps,1d}(m_\eps) \leq
    \limsup_{\eps \to 0} F_{\eps,1d}(m_\eps) \leq 2 F_0(n_0),
  \end{align}
  where $n_0 \in [0,1]$ is the minimizer of $F_0$ from Lemma
  \ref{l:F0}. The proof proceeds in four steps.

  \medskip
  
  \noindent \emph{Step 1: Construction of a test potential.} We first
  establish the liminf inequality in \eqref{F0infsup1d}. Focusing on
  the stray field energy, we use Lemma \ref{l:FepsSlb} with the test
  function $v \in \mathring{H}^1(\R^2)$ constructed as follows. For
  $n_\eps := m_{\eps,2}(\delta_\eps / L_\eps)$, define
  \begin{align}
    \label{v1-1d}
    v_1(\rho) :=
    \begin{cases}
      -{n_\eps \over 2 \pi} \ln \left( { b \over 2 \eps \delta_\eps}
      \right), & 0 \leq \rho \leq \delta_\eps / L_\eps , \\
      -{n_\eps \over 2 \pi} \ln \left( { b \over 2 \eps L_\eps \rho}
      \right), & \delta_\eps / L_\eps \leq \rho \leq b / (2 \eps
      L_\eps), \\
      0 , & \rho \geq b / (2 \eps L_\eps),
    \end{cases}
  \end{align}
  and
  \begin{align}
    \label{v2-1d}
    v_2(\rho) :=
    \begin{cases}
      {n_\eps - 1 \over 2 \pi} \ln \left( { b \over 2 \eps
          L_\eps} \right), & 0 \leq \rho \leq 1 , \\
      {n_\eps - 1 \over 2 \pi} \ln \left( { b \over 2 \eps L_\eps
          \rho} \right), & 1 \leq \rho \leq b / (2
      \eps L_\eps), \\
      0 , & \rho \geq b / (2 \eps L_\eps),
    \end{cases}
  \end{align}
  We then define, for all $(x, z) \in \R^2$, the test potential
  \begin{multline}
    \label{v-1d}
    v(x, z) := v_1\left( \sqrt{x^2 + z^2} \right) + v_2 \left(
      \sqrt{x^2 + z^2} \right) \\
    - v_1 \left( \sqrt{(\eps^{-1} L_\eps^{-1} b - x)^2 + z^2} \right)
    - v_2 \left( \sqrt{(\eps^{-1} L_\eps^{-1} b - x)^2 + z^2} \right).
  \end{multline}
  Clearly, $v$ is admissible. Furthermore, in view of the symmetry of
  $m_\eps$ guaranteed by Theorem \ref{t:1d} we have
  \begin{align}
    \label{v1dtracesymm}
    \int_0^{\eps^{-1} L_\eps^{-1} b} v(x, 0) \left( m_{\eps,2}(x) \widetilde
    \eta_{\delta_\eps / L_\eps} (x) \right)' \d x = 2 \int_0^{\eps^{-1}
    L_\eps^{-1} b/2} v(x, 0) \left( m_{\eps,2}(x) \widetilde 
    \eta_{\delta_\eps / L_\eps} (x) \right)' \d x. 
  \end{align}
  Similarly, we have
  \begin{align}
    \label{gradvsymm}
    \int_\R \int_\R |\nabla v(x, z)|^2 \d x \d z = 2 \int_\R
    \int_{-\infty}^{\eps^{-1} L_\eps^{-1} b / 2}
    |\nabla v(x, z)|^2 \d x \d z \notag \\
   = 4 \pi \int_0^{\eps^{-1} L_\eps^{-1} b / 2} |\nabla v_1(\rho) +
    \nabla v_2(\rho)|^2 \rho \d \rho.
  \end{align}
  Carrying out the integration in polar coordinates yields
  \begin{align}
    \label{gradest}
    \int_\R \int_\R |\nabla v(x, z)|^2 \d x \d z = 4 \pi
    \int_{\delta_\eps / L_\eps}^1 |\nabla v_1(\rho)|^2 \rho \d \rho + 4 \pi
    \int_1^{\eps^{-1} L_\eps^{-1} b/2} |\nabla v_1(\rho) + \nabla
    v_2(\rho)|^2 \rho \d \rho \notag \\
    = {n_\eps^2 \over \pi} \ln \left( {L_\eps \over \delta_\eps}
    \right) + {1 \over \pi} \ln \left( {b \over 2 \eps L_\eps}
    \right).  
  \end{align}

  \medskip
  
  \noindent \emph{Step 2: Computation of the potential energy.} We now 
  write, using the definition of the potential $v$ in \eqref{v-1d}:
  \begin{multline}
    \label{m2v12}
    \int_0^{\eps^{-1}L_\eps^{-1} b/2} v(x, 0) \left( m_{\eps,2}(x)
      \widetilde \eta_{\delta_\eps / L_\eps} (x) \right)' \d x \\
    = \int_0^{\eps^{-1} L_\eps^{-1} b/2} v_1(x) \left( m_{\eps,2}(x)
      \widetilde \eta_{\delta_\eps / L_\eps} (x) \right)' \d x +
    \int_0^{\eps^{-1} L_\eps^{-1} b/2} v_2(x) \left( m_{\eps,2}(x)
      \widetilde \eta_{\delta_\eps / L_\eps} (x) \right)' \d x.
  \end{multline}
  With the help of the definition of $v_1$ in \eqref{v1-1d}, we have
  for the first term in the right-hand side of \eqref{m2v12}:
  \begin{align}
    \int_0^{\eps^{-1} L_\eps^{-1} b/2} v_1(x) \left( m_{\eps,2}(x)
    \widetilde \eta_{\delta_\eps / L_\eps} (x) \right)' \d x =
    v_1(0) n_\eps + \int_{\delta_\eps / L_\eps}^{\eps^{-1} L_\eps^{-1}
    b/2}  v_1(x) m_{\eps,2}'(x) \d x \notag \\
    =  v_1(0) n_\eps + \int_{\delta_\eps / L_\eps}^1  v_1(x)
    m_{\eps,2}'(x) \d x + \int_1^{\eps^{-1}
    L_\eps^{-1} b/2}  v_1(x) m_{\eps,2}'(x) \d x \notag \\
    = (v_1(0) - v_1(1)) n_\eps + v_1(1) m_{\eps,2}(1) +
    \int_{\delta_\eps / L_\eps}^1 (v_1(x) - v_1(1)) m_{\eps,2}'(x) \d
    x \notag \\
    + \int_1^{\eps^{-1}  L_\eps^{-1} b/2}  v_1(x) m_{\eps,2}'(x) \d x
    \notag \\
    =  (v_1(0) - v_1(1)) n_\eps + v_1(1) +
    \int_{\delta_\eps / L_\eps}^1 (v_1(x) - v_1(1)) m_{\eps,2}'(x) \d
    x  \notag \\
    + \int_1^{\eps^{-1}  L_\eps^{-1} b/2}  v_1'(x) (1 - m_{\eps,2}(x))
    \d x, 
  \end{align}
  where in the last line we used integration by parts. Similarly, with
  the help of the definition of $v_2$ in \eqref{v2-1d} we have for the
  second term in the right-hand side of \eqref{m2v12}:
  \begin{align}
    \int_0^{\eps^{-1} L_\eps^{-1} b/2} v_2(x) \left( m_{\eps,2}(x)
    \widetilde \eta_{\delta_\eps / L_\eps} (x) \right)' \d x =
    v_2(1) m_{\eps,2}(1) + \int_1^{\eps^{-1} L_\eps^{-1}
    b/2}  v_2(x) m_{\eps,2}'(x) \d x \notag \\
    =  v_2(1) + \int_1^{\eps^{-1}  L_\eps^{-1} b/2}  v_2'(x) (1 -
    m_{\eps,2}(x)) \d x, 
  \end{align}
  again, using integration by parts. Combining the two formulas above
  yields
  \begin{multline}
    \int_0^{\eps^{-1}L_\eps^{-1} b/2} v(x, 0) \left( m_{\eps,2}(x)
      \widetilde \eta_{\delta_\eps / L_\eps} (x) \right)' \d x =
    v_1(1)
    + v_2(1) + (v_1(0) - v_1(1)) n_\eps \\
    + \int_{\delta_\eps / L_\eps}^1 (v_1(x) - v_1(1)) m_{\eps,2}'(x)
    \d x + \int_1^{\eps^{-1} L_\eps^{-1} b/2} (v_1'(x) + v_2'(x)) (1 -
    m_{\eps,2}(x)) \d x.
  \end{multline}
  Finally, recalling the precise definitions of $v_1$ and $v_2$, we
  obtain
  \begin{align}
    \label{vpotenergy}
    \int_0^{\eps^{-1}L_\eps^{-1} b/2} v(x, 0) \left( m_{\eps,2}(x)
      \widetilde \eta_{\delta_\eps / L_\eps} (x) \right)' \d x = -{1
    \over 2 \pi} \ln \left( {b \over 2 \eps L_\eps} \right) -
    {n_\eps^2 \over 2 \pi} \ln \left( {L_\eps \over \delta_\eps}
    \right) \notag \\
    + {n_\eps \over 2 \pi}  \int_{\delta_\eps / L_\eps}^1
    m_{\eps,2}'(x) \ln x
    \d x +  {1 \over 2 \pi}  \int_1^{\eps^{-1} L_\eps^{-1} b/2} {1 -
    m_{\eps,2}(x) \over x} \d x.
  \end{align}
  
  \medskip
  
  \noindent \emph{Step 3: Lower bound.} We now estimate the left-hand
  side of \eqref{vpotenergy}, using Young's inequality:
  \begin{align}\label{bubu}
    \int_0^{\eps^{-1}L_\eps^{-1} b/2} v(x, 0) \left( m_{\eps,2}(x)
    \widetilde \eta_{\delta_\eps / L_\eps} (x) \right)' \d x \leq -{1
    \over 2 \pi} \ln \left( {b \over 2 \eps L_\eps} \right) -
    {n_\eps^2 \over 2 \pi} \ln \left( {L_\eps \over \delta_\eps}
    \right) \notag \\
    + {1 \over 4 \pi} \int_{\delta_\eps / L_\eps}^1 \left( \ln^2
    x + |m_{2,\eps}'(x)|^2 \right) \d x
    + {1 \over 2 \pi} \int_1^{\eps^{-1} L_\eps^{-1} b/2}
    (1 - m_{\eps,2}(x)) \d x \notag \\
    \leq -{1
    \over 2 \pi} \ln \left( {b \over 2 \eps L_\eps} \right) -
    {n_\eps^2 \over 2 \pi} \ln \left( {L_\eps \over \delta_\eps}
    \right) + C \left( 1 +
    F_{\eps,1d}^{MM}(m_\eps) \right),
  \end{align}
  for some $C > 0$ independent of $\eps \ll 1$. Thus, according to
  Lemma \ref{l:FepsSlb} and \eqref{gradest}, we have
  \begin{align}
    F_{\eps,1d}^S(m_\eps) \geq -{\lambda |\ln \eps| \over 2 \pi \ln
    |\ln \eps|} + {\lambda \over 2 \pi \ln |\ln \eps|} \ln \left( {b
    \over 2 \eps L_\eps} \right) + {\lambda n_\eps^2 \over 2 \pi \ln
    |\ln \eps|} \ln \left( {L_\eps \over \delta_\eps} \right)  \notag
    \\
    - {C \over \ln |\ln \eps|} \left( 1
    + F_{\eps,1d}^{MM}(m_\eps) \right),
  \end{align}
  again, for some $C > 0$ independent of $\eps \ll 1$. Recalling
  \eqref{lamdeleps} and \eqref{Leps}, this translates into
  \begin{align}
    \label{Feps1dSlb}
    F_{\eps,1d}^S(m_\eps) \geq - {\lambda \over 2 \pi} + {\lambda
    n_\eps^2 \over \pi} - {C \over \ln |\ln \eps|} \left(
    1 + F_{\eps,1d}^{MM}(m_\eps)
    \right).  
  \end{align}
  Therefore, for any $\alpha \in (0, \frac12)$ and all $\eps$ small
  enough we can write
  \begin{align}
    \label{Fepslbal}
    F_{\eps,1d}(m_\eps) = F_{\eps,1d}^{MM}(m_\eps) +
    F_{\eps,1d}^S(m_\eps) \geq (1 - \alpha) \left[ F_{\eps,1d}^{MM}(m_\eps)
    + {\lambda \over 2 \pi}  (2 n_\eps^2 - 1) \right] - C \alpha,
  \end{align}
  for some $C > 0$ independent of $\eps$ and $\alpha$.

  Now, applying Lemma \ref{l:MM} we arrive at
  \begin{align}
    \label{FepsMMlb2}
    F_{\eps,1d}(m_\eps) \geq 2 (1 - \alpha) F_0(n_\eps) - 8
    \sqrt{\beta} \left( 1- \sqrt{1 + m_{2,\eps}(R) \over 2} \right) -
    C \alpha,
  \end{align}
  for any $R \in (0, \eps^{-1} L_\eps^{-1} b/2]$ and $C > 0$
  independent of $\eps \ll 1$, $\alpha$ and $R$. At the same time,
  using Lemma \ref{l:upe2} and \eqref{FepsMMlb2} we obtain
  \begin{align}
    \beta \int_0^{\eps^{-1} L_\eps^{-1} b} (1 - m_{2,\eps}) \d x \leq
    F_{\eps,1d}^{MM}(m_\eps) \leq C,
  \end{align}
  for some $C > 0$ and all $\eps \ll 1$. Therefore, there exists
  $R_\eps \in [\eps^{-1} L_\eps^{-1} b/4, \eps^{-1} L_\eps^{-1} b/2]$
  such that, choosing $R = R_\eps$ we have
  $m_{\eps,2}(R) \geq 1 - 4 C \eps L_\eps / (\beta b)$, so that
  the next-to-last term in \eqref{FepsMMlb2} can be absorbed into the
  last term. Thus, we have
  \begin{align}
    F_{\eps,1d}(m_\eps) \geq 2 (1 - \alpha) \min_{n \in [0,1]} F_0(n) -
    C \alpha,
  \end{align}
  for some $C > 0$ independent of $\alpha$ and $\eps$, for all $\eps$
  small enough, and the liminf inequality follows by sending
  $\alpha \to 0$.
  
  \medskip
  
  \noindent \emph{Step 4: Upper bound.}  Finally, we derive an
  asymptotically matching upper bound for the energy. We use the
  truncated optimal Modica-Mortola profile at the edges as a test
  configuration. More precisely, for $K > 1$ and $\eps$ sufficiently
  small, we define $m \in \mathcal A_\eps^{1d}$ satisfying
  \eqref{mthcossin} with
  $\theta(x) = \bar \theta(\min (x, b / (\eps L_\eps) - x))$, where
  \begin{align}
    \label{MMtrunc}
    \bar\theta(x) := 
    \begin{cases}
      \theta_0, & 0 \leq x \leq {\delta_\eps \over L_\eps}, \\
      4 \arctan \left( e^{-2 \sqrt{\beta} \left( x - {\delta_\eps
              \over L_\eps} \right)} \tan {\theta_0 \over 4} \right),
      & {\delta_\eps \over L_\eps} \leq
      x \leq K +  {\delta_\eps \over L_\eps}, \\
      4 \arctan \left( e^{-2 K \sqrt{\beta}} \tan {\theta_0 \over 4}
      \right) \left[ 1 - \eta \left( {x \over K} - 1 - {\delta_\eps
            \over K L_\eps} \right) \right], & K + {\delta_\eps \over
        L_\eps}
      \leq x \leq 2 K + {\delta_\eps \over L_\eps}, \\
      0, & 2 K + {\delta_\eps \over L_\eps}\leq x \leq {b \over 2 \eps
        L_\eps},
    \end{cases}
  \end{align}
  and $\theta_0 = \arccos n_0$, where $n_0$ is the unique minimizer of
  $F_0(n)$ in Lemma \ref{l:F0}. By the argument leading to the case of
  equality in \eqref{E1dthMM}, we obtain
  \begin{align}
    \label{F1dMMtest}
    F_{\eps,1d}^{MM}(m)
    & = 8 \sqrt{\beta}\left( 1 - \sqrt{1 + n_0 \over
      2} \right) - 8 \sqrt{\beta}\left( 1 - \sqrt{1 + \cos \left(
      \bar\theta \big( K +
      {\delta_\eps \over L_\eps} \big) \right) \over
      2} \right) \notag \\
    & + 2 \beta (1 - n_0) {\delta_\eps \over L_\eps} + \int_{\delta_\eps
      \over L_\eps}^{K + {\delta_\eps \over L_\eps}} \left( |\bar
      \theta'|^2 + 2 \beta (1 - \cos \bar \theta) \right) \d x.
  \end{align}
  Thus, we have
  \begin{align}
    \label{F1dMMtest2}
    F_{\eps,1d}^{MM}(m) \leq 8 \sqrt{\beta}\left( 1 - \sqrt{1 + n_0 \over
    2} \right) + {C \ln |\ln \eps| \over |\ln \eps|^2} + C K e^{-4 K
    \sqrt{\beta}},
  \end{align}
  for some $C > 0$ independent of $\eps$ and $K$ and all $\eps$
  sufficiently small.

  Turning now to the stray field energy, with the help of Lemma
  \ref{l:Str} we can write
  \begin{align}
    F_{\eps,1d}^S(m)
    & =  - {\lambda |\ln \eps| \over 2 \pi \ln |\ln
      \eps|} \notag \\
    & + {\lambda \over 4 \pi \ln |\ln \eps|} \int_0^{b
      \over \eps L_\eps} \int_0^{b \over \eps L_\eps} \ln {|x -
      y|^{-1} \over \eps L_\eps }
      \left( m_2(x) \widetilde \eta_{\delta_\eps / L_\eps}(x) \right)'
      \left( m_2(y) \widetilde \eta_{\delta_\eps / L_\eps}(y) \right)'
      \d x \d y,
  \end{align}
  where we took into account that inserting a constant factor to the
  argument of the logarithm does not change the stray field energy.
  With the help of the definition of $m$, this is equivalent to
  \begin{align}
    \label{FepsSlnln}
    F_{\eps,1d}^S(m)
    & =  - {\lambda |\ln \eps| \over 2 \pi \ln |\ln
      \eps|} \notag \\
    & + {\lambda \over 2 \pi \ln |\ln \eps|} \int_0^{2 K +
      {\delta_\eps \over L_\eps}} \int_0^{2 K +
      {\delta_\eps \over L_\eps}} \ln {|x
      - y|^{-1} \over \eps L_\eps }
      \left( m_2(x) \widetilde \eta_{\delta_\eps / L_\eps}(x) \right)'
      \left( m_2(y) \widetilde \eta_{\delta_\eps / L_\eps}(y) \right)'
      \d x \d y \notag \\
    & +  {\lambda \over 2 \pi \ln |\ln \eps|} \int_0^{2 K +
      {\delta_\eps \over L_\eps}} \int_{{b \over \eps L_\eps} - 2 K -
      {\delta_\eps \over L_\eps}}^{b \over \eps L_\eps} \ln {|x -
      y|^{-1} \over \eps L_\eps }
      \left( m_2(x) \widetilde \eta_{\delta_\eps / L_\eps}(x) \right)'
      \left( m_2(y) \widetilde \eta_{\delta_\eps / L_\eps}(y) \right)'
      \d x \d y.
  \end{align}
  Observe that the integral in the last line of \eqref{FepsSlnln} is
  bounded above by a constant independent of $\eps$ and $K$ for all
  $\eps$ sufficiently small. Therefore, we now concentrate on
  estimating the remaining terms in \eqref{FepsSlnln}.

  We can write the integral in the second line  in \eqref{FepsSlnln}
  as follows:
  \begin{align}
    \int_0^{2 K + {\delta_\eps \over L_\eps}}
    & \int_0^{2 K +
      {\delta_\eps \over L_\eps}}
      \ln {|x - y|^{-1} \over \eps L_\eps
      } \left( m_2(x) \widetilde \eta_{\delta_\eps / L_\eps}(x) \right)'
      \left( m_2(y) \widetilde \eta_{\delta_\eps / L_\eps}(y) \right)'
      \d x \d y \notag \\
    & = n_0^2 \int_0^{\delta_\eps \over L_\eps} \int_0^{\delta_\eps
      \over L_\eps} \ln {|x - y|^{-1} \over \eps L_\eps } 
      \widetilde \eta_{\delta_\eps / L_\eps}'(x)
      \widetilde \eta_{\delta_\eps / L_\eps}'(y) \d x \d y \notag \\
    & + 2 n_0 \int_{\delta_\eps \over
      L_\eps}^{2 K + {\delta_\eps \over L_\eps}} \int_0^{\delta_\eps
      \over L_\eps} \ln {|x - y|^{-1} 
      \over \eps L_\eps } \widetilde \eta_{\delta_\eps / L_\eps}'(x)
      m_2'(y) \d x \d y \notag \\
    & + \int_{\delta_\eps \over
      L_\eps}^{2 K + {\delta_\eps \over L_\eps}} \int_{\delta_\eps \over
      L_\eps}^{2 K + {\delta_\eps \over L_\eps}} \ln {|x - y|^{-1}
      \over \eps L_\eps } m_2'(x) m_2'(y) \d x \d y =: I_1 + I_2 + I_3.
  \end{align}
  For the first integral, we have
  \begin{align}
    I_1 = n_0^2 \ln {1 \over \eps \delta_\eps} + n_0^2 \int_0^1
    \int_0^1 \ln |x - y|^{-1} \eta'(x)  \eta'(y) \d x \d y \leq n_0^2
    \ln {1 \over \eps \delta_\eps} + C,
  \end{align}
  for some $C > 0$ independent of $\eps$ and $K$ and all $\eps$
  sufficiently small. At the same time, noting that
  $m_2'(x + \delta_\eps L_\eps^{-1}) \geq 0$ for all $0 < x < 2 K$, we
  get
  \begin{align}
    I_2 \leq 2 n_0 (1 - n_0) \ln {1 \over \eps L_\eps} + 2 n_0
    \int_0^{2 K} \ln |y|^{-1} m_2'(y + \delta_\eps L_\eps^{-1}) \d y
    \leq 2 n_0 (1 - n_0) \ln {1 \over \eps L_\eps} + C, 
  \end{align}
  again, for some $C > 0$ independent of $\eps$ and $K$ and all $\eps$
  sufficiently small. Finally, for the third integral we have
  \begin{align}
    I_3 = (1 - n_0)^2 \ln {1 \over \eps L_\eps} + \int_{\delta_\eps \over
    L_\eps}^{2 K + {\delta_\eps \over L_\eps}} \int_{\delta_\eps \over
    L_\eps}^{2 K + {\delta_\eps \over L_\eps}} \ln |x - y|^{-1}
    m_2'(x) m_2'(y) \d x \d y \notag \\
    \leq (1 - n_0)^2 \ln {1 \over \eps L_\eps} + C K,
  \end{align}
  once again, for some $C > 0$ independent of $\eps$ and $K$ and all
  $\eps$ sufficiently small.

  We now put all the obtained estimates together:
  \begin{align}
    \label{I123up}
    I_1 + I_2 + I_3 \leq \ln {1 \over \eps L_\eps} + n_0^2 \ln {L_\eps
    \over \delta_\eps} + C K.
  \end{align}
  Then, recalling the definitions in \eqref{lamdeleps} and
  \eqref{Leps} and combining the estimate in \eqref{I123up} with the
  one in \eqref{F1dMMtest2}, we arrive at
  \begin{align}
    F_{\eps,1d}(m) \leq 8 \sqrt{\beta}\left( 1 - \sqrt{1 + n_0 \over
    2} \right) + {\lambda \over 2 \pi} (2 n_0^2 - 1) + {C \ln |\ln
    \eps| \over |\ln \eps|^2} + C K e^{-4 K \sqrt{\beta}} + {C K \over
    \ln |\ln \eps|}, 
  \end{align}
  for some $C > 0$ independent of $\eps$ and $K$ and all $\eps$
  sufficiently small. Taking the limsup as $\eps \to 0$, therefore,
  yields
  \begin{align}
    \limsup_{\eps \to 0} F_{\eps,1d}(m) \leq 2 F_0(n_0) +  C K e^{-4 K
    \sqrt{\beta}}. 
  \end{align}
  Finally, the result follows by sending $K \to \infty$.
\end{proof}

\begin{proof}[Proof of Theorem \ref{t:1dmin}]
  As in the proof of Theorem \ref{t:E1deps}, we consider minimizers
  $m_\eps$ of $F_{\eps,1d}$ and write $F_{\eps,1d}^{MM}(m_\eps)$ in
  the form
  \begin{align}
    \label{Feps1dMMth}
    F_{\eps,1d}^{MM}(m_\eps) = \int_0^{\eps^{-1} L_\eps^{-1}
    b} \left( \frac12 |\theta_\eps'|^2 + \beta (1 - \cos \theta_\eps)
    \right) \d x.
  \end{align}
  Also, without loss of generality we may assume that
  $\theta_\eps \geq 0$. Then with the help of the estimate in
  \eqref{Feps1dSlb} we can write
  \begin{align}
    \label{Feps1dMM12}
    F_{\eps,1d}(m_\eps) \geq \frac12 F_{\eps,1d}^{MM}(m_\eps) - C,
  \end{align}
  for some $C > 0$ independent of $\eps \ll 1$. On the other hand, by
  \eqref{F0infsup1d} we know that the left-hand side of
  \eqref{Feps1dMM12} is bounded independently of $\eps \ll 1$, which,
  in turn, implies that
  \begin{align}
    \label{thepsub}
    \| \theta_\eps' \|_{L^2(0, \eps^{-1} L_\eps^{-1} b)} \leq C.     
  \end{align}

  Now, pick a sequence of $\eps_k \to 0$ as $k \to \infty$. Then, up
  to a subsequence (not relabeled) we have
  $\theta_{\eps_k} \rightharpoonup \bar\theta$ in
  $H^1_{loc}(\overline{\R}^+)$ and locally uniformly by the estimate
  in \eqref{thepsub}. At the same time, using \eqref{Fepslbal} and the
  Modica-Mortola trick \cite{modica87}, we have
  \begin{multline}
    \label{Feps1dthlbMM}
    F_{\eps_k,1d}(m_{\eps_k}) \geq (1 - \alpha) \Bigg[ 4 \sqrt{\beta}
    \int_0^{\eps_k^{-1} L_{\eps_k}^{-1} b/2} \sin \left(
      {\theta_{\eps_k} \over 2}
    \right) |\theta_{\eps_k}'| \d x \\
    + \int_0^{\eps_k^{-1} L_{\eps_k}^{-1} b/2} \left[ |\theta_{\eps_k}'|
      - 2 \sqrt{\beta} \sin \left( {\theta_{\eps_k} \over 2} \right)
    \right]^2 \d x + {\lambda \over 2 \pi} (2 n_{\eps_k}^2 - 1) \Bigg]
    - C \alpha,
  \end{multline}
  for some $C > 0$ and any $\alpha \in (0,\frac12)$, for all $k$
  sufficiently large. Here we used the reflection symmetry of the
  minimizers and defined
  $n_{\eps_k} := m_{\eps_k,2}(\delta_{\eps_k} / L_{\eps_k})$. As in
  the proof of Theorem \ref{t:E1deps}, we can find
  $R_k \in (\alpha^{-3}, 2 \alpha^{-3})$ such that
  $\theta_{\eps_k}(R_k) < \alpha$ for all $\alpha$ sufficiently
  small. Then from \eqref{Feps1dthlbMM} we obtain
  \begin{align}
    \label{Feps1dthlbMM2}
    F_{\eps_k,1d}(m_{\eps_k})
    & \geq (1 - \alpha) \Bigg[ -4 \sqrt{\beta}
      \int_0 ^{R_k} \sin \left(
      {\theta_{\eps_k} \over 2} \right) \theta_{\eps_k}' \d x + 4
      \sqrt{\beta} \int_{R_k}^{\eps_k^{-1} L_{\eps_k}^{-1} b/2} \sin
      \left( {\theta_{\eps_k} \over 2}
      \right) |\theta_{\eps_k}'| \d x \notag \\
    & \qquad + \int_0^{R_k} \left[
      \theta_{\eps_k}' 
      + 2 \sqrt{\beta} \sin \left( {\theta_{\eps_k} \over 2} \right)
      \right]^2 \d x + {\lambda \over 2 \pi} (2 n_{\eps_k}^2 - 1) \Bigg]
      - C \alpha \notag \\
    & = (1 - \alpha) \Bigg[ 2 F_0[n_{\eps_k}] - 8 \sqrt{\beta} \left(
      1 + \sqrt{1 + \cos 
      \theta_{\eps_k}(R_k) \over 2} \right) \notag \\
    & \qquad + 8 \sqrt{\beta} \left(
      1 + \sqrt{1 + \cos 
      \theta_{\eps_k}(0) \over 2} \right) - 8 \sqrt{\beta} \left(
      1 + \sqrt{1 + \cos 
      \theta_{\eps_k}(\delta_{\eps_k} / L_{\eps_k}) \over 2} \right)
      \notag \\ 
    & \qquad + 4 \sqrt{\beta}
      \int_{R_k}^{\eps_k^{-1} L_{\eps_k}^{-1} b/2} \sin \left(
      {\theta_{\eps_k} \over 2}
      \right) |\theta_{\eps_k}'| \d x \notag \\
    & \qquad + \int_0^{R_k} \left[
      \theta_{\eps_k}'  
      + 2 \sqrt{\beta} \sin \left( {\theta_{\eps_k} \over 2} \right)
      \right]^2 \d x \Bigg] - C \alpha.
  \end{align}
  In view of the definition of $R_k$ the term involving
  $\cos \theta_{\eps_k}(R_k)$ in \eqref{Feps1dthlbMM2} may be absorbed
  into the last term for all $\alpha$ sufficiently small. Similarly,
  by \eqref{thepsub} and Sobolev embedding the second line in the
  right-hand side of \eqref{Feps1dthlbMM2} may be bounded by
  $(\delta_{\eps_k} / L_{\eps_k})^{1/2}$ and, hence, absorbed into the
  last term as well for all $k$ sufficiently large depending on
  $\alpha$. Thus, taking into account that
  $F_{\eps_k,1d}(m_{\eps_k}) \to 2 F_0(n_0)$ as $k \to \infty$, we
  obtain for all $k$ large enough
  \begin{align}
    \label{Feps1dthlbMM3}
    (1 - \alpha)^{-1} F_0(n_0) + C \alpha 
    & \geq 
      F_0(n_{\eps_k}) + 2 \sqrt{\beta} 
      \int_{R_k}^{\eps_k^{-1} L_{\eps_k}^{-1} b/2} \sin \left(
      {\theta_{\eps_k} \over 2}
      \right) |\theta_{\eps_k}'| \d x \notag \\
    & \qquad + \frac12 \int_0^{R_k}
      \left[ \theta_{\eps_k}'  
      + 2 \sqrt{\beta} \sin \left( {\theta_{\eps_k} \over 2} \right)
      \right]^2 \d x.
  \end{align}

  We now observe that by minimality of $F_0(n_0)$ both integrals in
  the right-hand side of \eqref{Feps1dthlbMM3} are bounded above by
  $C \alpha$, for some $C > 0$ independent of $\alpha$ and $k$. In
  particular, this implies that the total variation of
  $\cos (\theta_{\eps_k}/2)$ on
  $(R_k, \frac12 \eps_k^{-1} L_{\eps_k}^{-1} b)$ is bounded by
  $C \alpha$, and in view of the fact that
  $\theta_{\eps_k}(R_k) < \alpha$ we conclude that
  $\theta_{\eps_k}(x) < C \alpha$ for all
  $x \in [2 \alpha^{-3}, \frac12 \eps_k^{-1} L_{\eps_k}^{-1} b]$ for
  some $C > 0$ independent of $\alpha$ and $k$ for all $k$
  sufficiently large. On the other hand, sending $\alpha \to 0$ on a
  sequence and extracting a further subsequence (not relabeled), we
  conclude that
  \begin{align}
    \label{thlimL2}
    \theta_{\eps_k}' + 2 \sqrt{\beta} \sin \left( {\theta_{\eps_k}
    \over 2} \right) \to 0 \ \text{in} \ L^2_{loc}(\overline{\R^+}),
  \end{align}
  as $k \to \infty$. Testing the left-hand side of \eqref{thlimL2}
  against $\phi \in C^\infty_c(\R^+)$ and passing to the limit, we
  then conclude that $\bar \theta$ satisfies
  \begin{align}
    \label{barthweak}
    {d \bar \theta \over dx} + 2 \sqrt{\beta} \sin \left( {\bar\theta
    \over 2} \right) = 0 \ \text{in} \ \mathcal D'(\R^+). 
  \end{align}
  In particular, since $\bar\theta \in C(\overline{\R^+})$, we have
  that $\bar\theta(x)$ also satisfies \eqref{barthweak} classically
  for all $x > 0$. Finally, by strict convexity of $F_0$ we can also
  conclude that $n_{\eps_k} \to n_0$ as $k \to \infty$. Therefore, we
  have
  \begin{align}
    \arccos n_0 = \lim_{k \to \infty} \arccos n_{\eps_k} = \lim_{k \to \infty} 
    \theta_{\eps_k}(\delta_{\eps_k} / L_{\eps_k}) 
    = \bar\theta(0).     
  \end{align}
  Thus, $\bar\theta = \theta_\infty$, where the latter is given by
  \eqref{thinf} with $\theta_0 = \arccos n_0$. Combining this with the
  uniform closeness of $\theta_{k_\eps}(x)$ to zero far from $x = 0$
  and the asymptotic decay of $\theta_\infty(x)$ for large $x> 0$ then
  yields uniform convergence of $\theta_{k_\eps}$ to $\theta_\infty$
  on $[0, \frac12 \eps_k^{-1} L_{\eps_k}^{-1} b]$. From
  \eqref{thlimL2} we conclude that this convergence is also strong in
  $H^1_{loc}(\overline{\R^+})$. Finally, in view of the uniqueness of
  $\bar\theta$ the limit does not depend on the choice of the
  subsequence and, hence, is a full limit.
\end{proof}

\section{Proof of Theorem \ref{t:rect}}
\label{sec:rect}

The proof follows closely the arguments in
Sec. \ref{sec:proof-theor-23}, except that we can no longer reduce the
problem to studying a one-dimensional profile due to lack of
translational symmetry in the $x_1$-direction. Therefore, we need to
incorporate the relevant corrections to the upper and lower bounds in
the proof of Theorem \ref{t:E1deps} and show that they are indeed
negligible in comparison with the limit energy $F_0$.

As in Sec. \ref{sec:proof-theor-23}, for
$\widetilde D_\eps := \eps^{-1} L_\eps^{-1} D$ and
$m \in \mathcal A_\eps$, where
\begin{align}
  \label{Aeps}
  \mathcal A_\eps := H^1(\widetilde D_\eps ; \mathbb S^1),
\end{align}
we introduce
\begin{align}
  \label{FepsDeps}
  F_\eps(m)
  & := \frac12 \int_{\widetilde D_\eps \cap \{ |m_2| < 1 \} }  {
    |\nabla m_2|^2 \over 1 - m_2^2} \d x + \beta
    \int_{\widetilde D_\eps} ( 1 - m_2) \d x 
    - {\lambda a \over 2 \pi \eps} \notag \\
  & \qquad\qquad + {\lambda \over 8 \pi \ln |\ln \eps|}
    \int_{\widetilde D_\eps} 
    \int_{\widetilde D_\eps} {\nabla \cdot \widetilde m_{\delta_\eps /
    L_\eps}(x) \nabla \cdot \widetilde m_{\delta_\eps / L_\eps}(y)
    \over |x - y|} \d x \d y,   
\end{align}
where
$\widetilde m_{\delta_\eps / L_\eps}(x) := m(x) \widetilde
\eta_{\delta_\eps / L_\eps} (x)$, with
$\widetilde \eta_{\delta_\eps / L_\eps} (x) := \eta( \text{dist}(x,
\partial \widetilde D_\eps) L_\eps / \delta_\eps)$. Then, for
$m \in H^1(D; \mathbb S^1)$ the connection between $F_\eps$ and the
original energy $E_\eps$ is given by
\begin{align}
  \label{FepsEeps}
  E_\eps(m) = {\lambda a \over 2 \pi} + \eps F_\eps(m(\cdot / (\eps
  L_\eps)),
\end{align}
which follows by a straightforward rescaling and applying the weak
chain rule \cite[Theorem 6.16]{lieb-loss} to the identity $|m|^2 =
1$. Therefore, the first statement of Theorem \ref{t:rect} is
equivalent to
\begin{align}
  \label{Fepsmin}
  {\eps |\ln \eps| \over \ln |\ln \eps|} \min_{m \in \mathcal A_\eps}
  F_\eps(m) \to 2 a \min_{n \in [0,1]} F_0(n) \quad \text{as} \quad
  \eps \to 0.  
\end{align}

As in the proof of Theorem \ref{t:E1deps}, we split the rescaled
energy into the local and the non-local parts
\begin{align}
  \label{FepsrectMMS}
   F_\eps(m) = F_\eps^{MM}(m) + F_\eps^S(m), 
\end{align}
where
\begin{align}
  \label{FepsMM2d}
  F_\eps^{MM}(m) := \frac12 \int_{\widetilde D_\eps \cap \{ |m_2| < 1
  \} } { |\nabla m_2|^2 \over 1 - m_2^2} \d x + \beta \int_{\widetilde
  D_\eps}   ( 1 - m_2) \d x,  
\end{align}
and 
\begin{align}
  \label{FepsS2d}
  F_\eps^S(m) :=  {\lambda \over 8 \pi \ln |\ln \eps|}
    \int_{\widetilde D_\eps} 
    \int_{\widetilde D_\eps} {\nabla \cdot \widetilde m_{\delta_\eps /
    L_\eps}(x) \nabla \cdot \widetilde m_{\delta_\eps / L_\eps}(y)
    \over |x - y|} \d x \d y - {\lambda a \over 2 \pi \eps}.
\end{align}

We begin by stating an analog of Lemma \ref{l:MM} in the case of a
rectangular domain.

\begin{lemma}
  \label{l:MM2d}
  Let $m = (m_1, m_2) \in \mathcal A_\eps$ and let
  $\overline m = (\overline m_1, \overline m_2)$ be defined as
  \begin{align}
    \label{mbar2d}
    \overline m_2(x_1, x_2) := {\eps L_\eps \over a} \int_0^{\eps^{-1}
    L_\eps^{-1} a} m_2(t, x_2) \, \d t, \qquad \overline m_1(x_1, x_2)
    := \sqrt{1 - \overline m_2^2(x_1, x_2)}.
  \end{align}
  Then $\overline m \in \mathcal A_\eps \cap C(\widetilde D_\eps)$,
  and for every $R \in (0, \eps^{-1} L_\eps^{-1} b/2]$ and every
  $r \in (0, R)$ there holds
  \begin{align}
    \label{FepsMMlb2d}
    {\eps L_\eps F_\eps^{MM}(m) \over a} \geq 4 \sqrt{\beta} \left. \left( 1
    - \sqrt{1 + \overline m_2 \over 2} \right) \right|_{x_2 = r}  + 4
    \sqrt{\beta} \left. \left( 1 - 
    \sqrt{1 + \overline m_2 \over 2}
    \right) \right|_{x_2 = \eps^{-1} L_\eps^{-1} b - r} \notag \\
    - 8 \sqrt{\beta} \left. \left( 1 - \sqrt{1 + \overline m_2
    \over 2} \right) \right|_{x_2 = R}.
  \end{align}
\end{lemma}

\begin{proof}
  Since $m \in H^1(\widetilde D_\eps; \mathbb S^1)$, its trace on
  $\widetilde D_\eps \cap \{x_2 = t\}$ is well-defined for every
  $t \in (0, \eps^{-1} L_\eps^{-1} b)$. Arguing by approximation, we
  have
  $\overline m_2 \in H^1(\widetilde D_\eps) \cap C(\widetilde
  D_\eps)$, in view of the one-dimensional character of
  $\overline m_2$. Furthermore, arguing exactly as in the proof of
  Lemma \ref{l:madmiss}, we also obtain that
  $\overline m \in \mathcal A_\eps \cap C(\widetilde D_\eps)$ and
  \begin{align}
    \label{mmbar2d}
    \int_{\widetilde D_\eps} |\nabla  m|^2 \d x \geq
    \int_{\widetilde D_\eps} |\nabla \overline m|^2 \d x =
    \int_{\widetilde D_\eps \cap \{ |\overline m_2| < 1\} } {|\nabla
    \overline m_2|^2 \over 1 - \overline m_2^2} \d x.
  \end{align}
  In particular, since $\overline m$ is independent of $x_1$, it may
  be chosen to be continuous in $\widetilde D_\eps$.

  By \eqref{mmbar2d} we have
  \begin{align}
    \label{FepsMM2da}
    F_\eps^{MM}(m) \geq \frac12 \int_{\widetilde D_\eps \cap \{
    |\overline m_2| < 1\} } {|\nabla \overline m_2|^2 \over 1 -
    \overline m_2^2} \d x + \beta  \int_{\widetilde D_\eps} (1 -
    \overline m_2) \d x.
  \end{align}
  Therefore, using the Modica-Mortola trick \cite{modica87}, for every
  $\delta \in (0,1)$ we obtain
  \begin{align}
    F_\eps^{MM}(m) \geq \frac12 \int_{\widetilde D_\eps \cap \{
    |\overline m_2| < 1\} } {|\nabla \overline m_2|^2 \over 1 -
    \overline m_2^2} \d x + \beta  \int_{\widetilde D_\eps \cap \{
    |\overline m_2| < 1\}} (1 -
    \overline m_2) \d x \notag \\
    \geq \sqrt{2 \beta} \int_{\widetilde D_\eps \cap \{
    \overline m_2 > -1 + \delta \} } {|\nabla \overline m_2| \over \sqrt{1 +
    \overline m_2}} \, \d x = \sqrt{8 \beta} \int_{\widetilde D_\eps}
    \left| \nabla \left( \sqrt{1 + \overline m_2^\delta} - \sqrt{2}
    \right) \right| \, \d x,
  \end{align}
  where
  $\overline m_2^\delta := \max(-1 + \delta, \overline m_2) \in
  H^1(\widetilde D_\eps)$ and we used weak chain rule \cite[Theorem
  6.16]{lieb-loss} and the fact that $\nabla \overline m_2^\delta = 0$
  on
  $\{ \overline m_2^\delta = -1 + \delta \} \cup \{ \overline
  m_2^\delta = 1\}$ \cite[Theorem 6.19]{lieb-loss}. Thus, in view of
  continuity of $\overline m_2^\delta$ we get (with a slight abuse of
  notation)
  \begin{align}
    {\eps L_\eps F_\eps^{MM}(m) \over a}
    &
      \geq \sqrt{8 \beta}
      \int_0^{\eps^{-1} L_\eps^{-1} a} \left| \left( \sqrt{2} -
      \sqrt{1 + \overline m_2^\delta(x_2)} \right)' \right| \, \d x_2 
      \geq  \sqrt{8 \beta} \int_r^R \left( \sqrt{2} - \sqrt{1 +
      \overline m_2^\delta(x_2)} \right)' \d x_2 
      \notag \\
    &
      - \sqrt{8 \beta} \int_R^{\eps^{-1} L_\eps^{-1} a - r} \left( \sqrt{2} -
      \sqrt{1 + \overline m_2^\delta(x_2)} \right)' \d x_2,
  \end{align}
  which yields the rest of the claim in view of arbitrariness of
  $\delta$.
\end{proof}

\paragraph{Lower bound for the stray field.}  In order to get the
required estimates for the lower bound, we have to extend the
definition of the test potential in a suitable way. Using the same
arguments as in the periodic case we have a similar lower bound for
the stray field energy:
\beq F^S_\eps(m) \geq - \frac{\lambda }{2 \ln |\ln \eps| }\int_{\R^3}
|\grad v|^2 \d x - \frac{\lambda }{\ln |\ln \eps|} \int_{{\widetilde
    D_\eps}} v(x_1,x_2,0) \div \widetilde m_{\delta_\eps /L_\eps} \d x -
{\lambda a \over 2 \pi \eps} \eeq for every
$v \in \mathring H^1 (\R^3)$.
We also define
\begin{align}
  \label{eq:nl}
  n^-_\eps & := \frac{\eps L_\eps}{a} \int_0^{a \eps^{-1} L_\eps^{-1}}
             \widetilde m_{\delta_\eps / L_\eps,2}(t, {\delta_\eps}/{L_\eps})
             \, \d t , \\
  \label{eq:np}
  n^+_\eps & := \frac{\eps L_\eps}{a} \int_0^{a
             \eps^{-1} L_\eps^{-1}} \widetilde m_{\delta_\eps /
             L_\eps, 2}(t,
             b/(\eps L_\eps) - {\delta_\eps}/{L_\eps}) \, \d t. 
\end{align}
The construction of the potential is done in the same way as before
with the only difference that we now do not have the reflection
symmetry for $\widetilde m_{\delta_\eps / L_\eps, 2}$ and have to
consider different distributions of charges near the bottom and the
top boundaries. We will carry out the calculation only near the bottom
boundary, the other calculation is completely analogous. To avoid
cumbersome notation, we will suppress the superscript ``$-$''
throughout the argument.

We would like to find a suitable test potential $v$ that vanishes for
$x_2 > b/(2\eps L_\eps)$ to obtain an appropriate asymptotic lower
bound. Let us define $v$ as follows: for
$x_1\in (0, \eps^{-1} L_\eps^{-1} a)$ we define
\begin{align}
  \label{eq:v12rect}
  v(x_1, x_2, x_3) := v_1\left( \sqrt{x_2^2 + x_3^2} \right) + v_2
  \left( \sqrt{x_2^2 + x_3^2} \right),  
\end{align}
where $v_1$ and $v_2$ are as in \eqref{v1-1d} and \eqref{v2-1d},
respectively, while for $x_1\in (-\infty, 0)$ we extend the definition
of $v$ as
  \begin{align}
    \label{v-2d-1}
    v(x_1, x_2, x_3) := v_1\left( \sqrt{x_1^2+ x_2^2 + x_3^2} \right)
    + v_2 \left(
      \sqrt{x_1^2 + x_2^2 + x_3^2} \right).
  \end{align}
  Finally, for $x_1\in (\eps^{-1} L_\eps^{-1} a, +\infty)$ we extend
  the definition of $v$ as
  \begin{align}
    \label{v-2d-2}
    v(x_1, x_2, x_3) := v_1\left( \sqrt{ \left( \eps^{-1}
    L_\eps^{-1}a - x_1 \right)^2+ x_2^2 + x_3^2} \right) + v_2 \left(
    \sqrt{ \left( \eps^{-1} L_\eps^{-1}a - x_1 \right)^2 + x_2^2 +
    x_3^2} \right)  
  \end{align}
  It is clear that $v \in \mathring H^1 (\R^3)$, and we can compute
  $I :=\int_{\R^3} |\grad v|^2 \d x$ explicitly. First, we split this
  integral into three parts: \beq I= \int_{(-\infty,0) \times \R^2}
  |\grad v|^2 \d x + \int_{(0,\eps^{-1} L_\eps^{-1} a) \times \R^2}
  |\grad v|^2 \d x +\int_{(\eps^{-1} L_\eps^{-1}a, +\infty) \times
    \R^2} |\grad v|^2 \d x.  \eeq It is clear that the first and the
  last integrals in the above expression coincide and the second
  integral was already essentially computed in \eqref{gradest}. Due to
  the symmetry of $v$ it is not difficult to see that \beq
  \int_{(-\infty,0) \times \R^2} |\grad v|^2 \d^3 x = 2 \pi
  \int_0^{\eps^{-1} L_\eps^{-1} b/2} \brackets{\pd{v_1}{r} +
    \pd{v_2}{r}}^2 r^2 \d r = \frac{1}{2\pi} \left( n_\eps^2 -1 +
    \frac{b}{2\eps L_\eps} - \frac{\delta_\eps n_\eps^2}{L_\eps}
  \right).  \eeq Therefore, we obtain \beq I = {a n_\eps^2 \over 2 \pi
    \eps L_\eps} \ln \left( {L_\eps \over \delta_\eps} \right) + {a
    \over 2 \pi \eps L_\eps} \ln \left( {b \over 2 \eps L_\eps}
  \right)+ \frac{1}{\pi} \left( n_\eps^2 -1 + \frac{b}{2\eps L_\eps} -
    \frac{\delta_\eps n_\eps^2}{L_\eps} \right).  \eeq

  Next we compute \beq J := \int_{{\widetilde D_\eps}} v(x_1,x_2,0) \,
  \div \widetilde m_{\delta_\eps/ L_\eps} \d x.  \eeq Note that for
  $x_1 \in (0,\eps^{-1} L_\eps^{-1}a)$ our function $v(x_1, x_2, 0)$
  depends only on $x_2$, and $\widetilde m_{\delta_\eps/L_\eps}$
  vanishes at the boundary of $\widetilde D_\eps$.  Therefore, with a
  slight abuse of notation we have
\begin{multline}
  J = \int_0^{\eps^{-1} L_\eps^{-1}a} \int_0^{\eps^{-1} L_\eps^{-1}
    b/2} v(x_2,0) \left(\partial_1 \widetilde m_{\delta_\eps/L_\eps,
      1}(x_1,x_2) + \partial_2 \widetilde m_{\delta_\eps/L_\eps,
      1}(x_1,x_2) \right) \, \d x_1 \d x_2 \\ = \frac{a}{\eps L_\eps}
  \int_0^{\eps^{-1} L_\eps^{-1} b} v(x_2,0) \partial_2 \overline
  m_{\delta_\eps/L_\eps, 2}(x_2) \, \d x_2,
\end{multline}
where
$\overline m_{\delta_\eps/L_\eps, 2}(x_2) := \frac{\eps L_\eps}{a}
\int_0^{\eps^{-1} L_\eps^{-1}a} \widetilde m_{\delta_\eps/L_\eps, 2}
(x_1,x_2) \, \d x_1$. Using the same arguments as for the periodic
case, we obtain a formula analogous to \eqref{vpotenergy}, with
$m_{\eps,2}$ replaced by $\overline m_{\delta_\eps/L_\eps, 2}$:
  \begin{align}
    \label{vpotenergy-rect} 
    \int_0^{\eps^{-1}L_\eps^{-1} b/2} v(x_2, 0) \left( \overline
    m_{\delta_\eps/L_\eps, 2} (x_2) \right)' \d x_2 = -{1 
    \over 2 \pi} \ln \left( {b \over 2 \eps L_\eps} \right) -
    {n_\eps^2 \over 2 \pi} \ln \left( {L_\eps \over \delta_\eps}
    \right) \notag \\
    + {n_\eps \over 2 \pi}  \int_{\delta_\eps / L_\eps}^1
    \left( \overline m_{\delta_\eps/L_\eps, 2} (x_2) \right)'  \ln x_2
    \d x_2 +  {1 \over 2 \pi}  \int_1^{\eps^{-1} L_\eps^{-1} b/2} {1 -
    \overline m_{\delta_\eps/L_\eps, 2} (x_2)  \over x_2} \d x_2.
  \end{align}
  
  We now would like to obtain an analog of \eqref{bubu} and need to
  estimate the last two terms in the right-hand side of
  \eqref{vpotenergy-rect}. The first term can be bounded as
  follows:
  \begin{align}
    \label{eq:n11}
    \Bigg| {n_\eps \over 2 \pi} \int_{\delta_\eps / L_\eps}^1
    & \left(
      \overline m_{\delta_\eps/L_\eps, 2} (x_2) \right)' \ln x_2 \d
      x_2 \Bigg| \leq \frac{\eps L_\eps}{2 \pi a} \left|
      \int_0^{\eps^{-1} L_\eps^{-1}a} 
      \int_{\delta_\eps / L_\eps}^1 \partial_{2} m_{2} (x_1, x_2)
      \widetilde
      \eta_{\delta_\eps / L_\eps} (x_1) \ln x_2 \d x_2 \, d x_1
      \right| \notag \\
    &\leq \frac{\eps L_\eps}{4 \pi a} \int_0^{\eps^{-1} L_\eps^{-1}a}
      \int_{\delta_\eps / 
      L_\eps}^1 \left(  |\partial_{2} m_{2} (x_1, x_2)|^2 + |\ln
      x_2|^2 \right) \d x_2
      \, d x_1 \leq \frac{\eps L_\eps}{2 \pi a} F_\eps^{MM} (m) + C,
  \end{align}
  for some universal $C > 0$.  Similarly, we can obtain
  \begin{align}
  \label{eq:n22}
    {1 \over 2 \pi}
    & \int_1^{\eps^{-1} L_\eps^{-1} b/2} {1 -
      \overline m_{\delta_\eps/L_\eps,2} (x_2)  \over x_2} \d x_2 =
      \frac{\eps L_\eps}{2 \pi a} 
      \int_0^{\eps^{-1} L_\eps^{-1}a}  \int_1^{\eps^{-1} L_\eps^{-1} b/2} {1
      - m_{2} (x_1, x_2) \widetilde 
      \eta_{\delta_\eps / L_\eps} (x_1) \over x_2}  \d x_2 \, \d x_1
      \notag \\     
    &\leq    \frac{\eps L_\eps}{2 \pi a}
      \int_{\delta_\eps / L_\eps}^{\eps^{-1} L_\eps^{-1}a - \delta_\eps / L_\eps}
      \int_1^{\eps^{-1} L_\eps^{-1} b/2} {1 - m_{2} (x_1, x_2) \over x_2}  \d
      x_2 \, \d x_1 + {2 \eps \delta_\eps \over \pi a} \ln \left( {b
      \over 2 \eps  
      L_\eps } \right) \notag \\
    & \qquad \leq \frac{\eps L_\eps}{2\pi \beta a}
      F_\eps^{MM} (m) + C, 
  \end{align}
  for some universal $C > 0$, provided that $\eps$ is small enough
  independently of $m$.  Thus, after some straightforward algebra we
  arrive at the following bound for $J$: \beq J \leq - \frac{a}{2 \pi
    \eps L_\eps} \left[ \ln \left( {b \over 2 \eps L_\eps} \right) +
    n_\eps^2 \ln \left( {L_\eps \over \delta_\eps} \right) \right] + C
  \left( \frac{1}{\eps L_\eps} + F_{\eps}^{MM}(m) \right), \eeq for
  some $C > 0$ and all $\eps$ small enough independent of $m$.
  
  Using the estimates for $I$ and $J$ above, and combining them with
  the estimates for the similarly defined potential that vanishes for
  $x_2 < b/(2\eps L_\eps)$, after some tedious algebra we obtain the
  following asymptotic lower bound for the stray field energy:
  \begin{align}
    F^S_\eps(m) \geq \frac{\lambda a \ln|\ln \eps|}{2 \pi \eps |\ln \eps|} 
    \Big(  |n_\eps^-|^2 +  |n_\eps^+ |^2 -1\Big)
    - {C \over \ln|\ln \eps|} \left( F_{\eps}^{MM}(m) +
    \frac{\ln  |\ln \eps| }{\eps |\ln \eps|} \right).   
  \end{align}

\paragraph{Upper bound for stray field.}
To derive an asymptotically sharp upper bound for the nonlocal energy,
we want to estimate from above the integral \beq W := {\lambda \over 8
  \pi \ln |\ln \eps|} \int_{\widetilde D_\eps} \int_{\widetilde
  D_\eps} {\nabla \cdot \widetilde m_{\delta_\eps / L_\eps}(x) \nabla
  \cdot \widetilde m_{\delta_\eps / L_\eps}(y) \over |x - y|} \d x \d
y, \eeq where
$\widetilde m_{\delta_\eps / L_\eps}(x) := m_\eps(x) \widetilde
\eta_{\delta_\eps / L_\eps}(x)$, and choose the test sequence
\begin{align}
  \label{eq:mepstest}
  m_\eps (x_1,x_2) := \left( \sqrt{1-m_{\eps,2}^2(x_2)},
  m_{\eps,2}(x_2) \right), 
\end{align}
in which $m_{\eps,2}$ is as defined by the one-dimensional
construction in Sec. \ref{sec:proof-theor-23}. We then obtain that
$W = {\lambda \over 8 \pi \ln |\ln \eps|} I$, where
\begin{multline}
  I = I_1 + I_2 + I_3 := \int_{\widetilde D_\eps} \int_{\widetilde
    D_\eps} \frac{\partial_1 \widetilde
    \eta_{\delta_\eps/L_\eps}(x_1,x_2) m_{\eps,1}(x_2) \ \partial_1
    \widetilde \eta_{\delta_\eps/L_\eps}(\xi_1,\xi_2)
    m_{\eps,1}(\xi_2) }{|x - \xi|} \d x \, \d \xi \\ + \int_{\widetilde
    D_\eps} \int_{\widetilde D_\eps} \frac{ \partial_2 (\widetilde
    \eta_{\delta_\eps/L_\eps}(x_1,x_2) m_{\eps,2}(x_2))\
    \partial_2(\widetilde \eta_{\delta_\eps/L_\eps}(\xi_1,\xi_2)
    m_{\eps,2}(\xi_2))}{|x - \xi|} \d x \, \d \xi \\+ 2
  \int_{\widetilde D_\eps} \int_{\widetilde D_\eps} \frac{\partial_1
    \widetilde \eta_{\delta_\eps/L_\eps}(x_1,x_2) m_{\eps,1}(x_2) \
    \partial_2(\widetilde \eta_{\delta_\eps/L_\eps}(\xi_1,\xi_2)
    m_{\eps,2}(\xi_2))}{|x - \xi|} \d x \, \d \xi.
\end{multline}
We see that the middle integral $I_2$ is asymptotically equivalent to
the one computed in the periodic case. Therefore, it is enough to
estimate the first and the last integrals and show that they only give
a negligible contribution into the stray field energy in the limit.

\begin{figure}
  \centering
  \includegraphics[width=8.5cm]{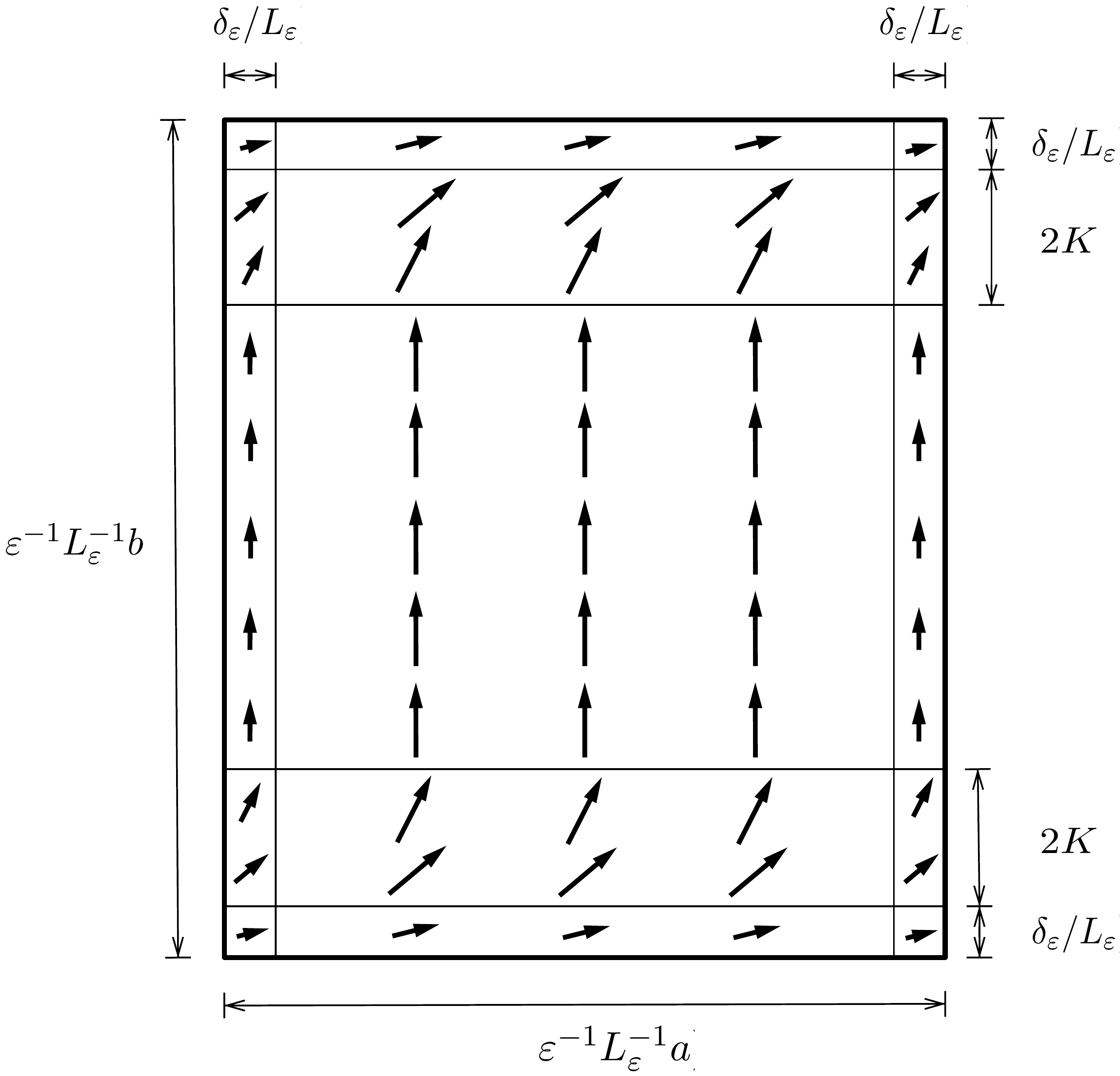}
  \caption{Schematics of the truncated test magnetization
    configuration $\widetilde m_{\delta_\eps / L_\eps}$ used in the
    upper bound construction of Sec. \ref{sec:rect}.}
  \label{fig:constr}
\end{figure}

We now estimate the first intergal $I_1$. Using the definition of
$\widetilde \eta_{\delta_\eps/L_\eps}$ we obtain that
$\partial_1 \widetilde \eta_{\delta_\eps/L_\eps}(x_1,x_2) =0$ for
$ \delta_\eps L_\eps^{-1} < x_1 < \eps^{-1} L_\eps^{-1} b -
\delta_\eps L_\eps^{-1}$. Moreover, outside this interval
$|\partial_1 \widetilde \eta_{\delta_\eps/L_\eps}(x_1,x_2)| \leq
L_\eps/\delta_\eps$. We also know that $m_{\eps,1}(x_2) = 0$ for
$x_2 \in (0, \delta_\eps/L_\eps) \cup (2K+ \delta_\eps/L_\eps,
\eps^{-1} L_\eps^{-1} b - 2K - \delta_\eps/L_\eps) \cup (\eps^{-1}
L_\eps^{-1} b - \delta_\eps/L_\eps)$, where $K$ is the same constant
as in the one-dimensional construction. Therefore, by direct
computation we can estimate for all $\eps$ sufficiently small
\begin{multline}
  I_1 \leq C \left( \frac{L_\eps}{\delta_\eps} \right)^2
  \int_{0}^{2K+\delta_\eps L_\eps^{-1} } \int_{0}^{2K+\delta_\eps
    L_\eps^{-1}} \int_0^{\delta_\eps L_\eps^{-1}} \int_0^{\delta_\eps
    L_\eps^{-1}} {\d x _1 \d \xi_1 \d x_2 \d \xi_2 \over \sqrt{(x_1 -
      \xi_1)^2 + (x_2 - \xi_2)^2}} \leq CK \ln
  \left(\frac{L_\eps}{\delta_\eps}\right),
\end{multline}
for some universal $C > 0$.  Similarly, the last integral $I_3$ can be
estimated as
\begin{multline}
  I_3 \leq C \left( \frac{L_\eps}{\delta_\eps} \right)
  \int_{\delta_\eps L_\eps^{-1} }^{2K + \delta_\eps L_\eps^{-1} }
  \int_0^{\eps^{-1} L_\eps^{-1} a} \int_0^{2K + \delta_\eps
    L_\eps^{-1} } \int_0^{\delta_\eps L_\eps^{-1}} \frac{\d x_1 \d
    x_2 \d \xi_1 \d \xi_2 }{\sqrt{(x_1-\xi_1)^2 + (x_2-\xi_2)^2}}  \\
  + \left( \frac{L_\eps}{\delta_\eps} \right)^2 \int_{0}^{\delta_\eps
    L_\eps^{-1} } \int_0^{\eps^{-1} L_\eps^{-1} a} \int_0^{2K +
    \delta_\eps L_\eps^{-1} } \int_0^{\delta_\eps L_\eps^{-1}}
  \frac{\d x_1 \d x_2 \d \xi_1 \d \xi_2}{\sqrt{(x_1-\xi_1)^2 +
      (x_2-\xi_2)^2}} \leq C K \ln\left( \frac{1}{\eps}
  \right),
\end{multline}
again, for some universal $C > 0$ and all $\eps$ small enough.

\begin{proof}[Proof of Theorem \ref{t:rect}.] 
  We can combine the lower bounds for $F_\eps^{MM}$ and $F_\eps^S$ and
  proceed in the same way as in the one-dimensional case. There is a
  slight mismatch, as the definition of $n_\eps^\pm$ uses the average
  of $\widetilde m_{\delta_\eps/L_\eps, 2}$, while the lower bound
  \eqref{FepsMMlb2d} for $F_\eps^{MM}$ uses $\overline m_{\eps,
    2}$. However, we observe that
  \begin{multline} \label{use:ineq} \left| \frac{\eps L_\eps}{a}
      \int_0^{\eps^{-1} L_\eps^{-1}a} \widetilde
      m_{\delta_\eps/L_\eps, 2} (x_1, x_2)
      \d x_1 - \overline m_{\eps, 2} (x_2) \right| \\
    \leq \frac{\eps L_\eps}{a} \int_0^{\eps^{-1} L_\eps^{-1}a}
    |m_{\eps, 2}(x_1,x_2)| (1-\tilde \eta_{\delta_\eps/L_\eps}
    (x_1,x_2)) \, \d x_1 \leq C \eps \delta_\eps,
  \end{multline}
  for some $C > 0$ independent of $\eps$, and, therefore,
  asymptotically we can interchange the average of
  $\widetilde m_{\delta_\eps/L_\eps,2}$ with $\overline m_{\eps, 2}$
  in the formula in \eqref{FepsMMlb2d} and arrive at the full lower
  bound as in the one-dimensional case.  Using in addition the upper
  bound construction, the proof of \eqref{minEeps1d2} follows exactly
  as in the proof of Theorem \ref{t:E1deps} with the help of Lemma
  \ref{l:MM2d}. Convergence of $m_\eps$ to $e_2$ trivially follows
  from positivity of the stray field energy and boundedness of
  $E_\eps(m_\eps)$ as $\eps \to 0$.
  
  Assuming $m_\eps$ is a minimizer of $E_\eps$, in the same way as in
  the proof of the Theorem \ref{t:E1deps} it follows that
  $n_\eps^- \to n_0$ and $n_\eps^+ \to n_0$, therefore we have
  \begin{align}
    \overline m_{\delta_\eps/L_\eps,2} ({\delta_\eps}/{L_\eps}) \to n_0
    \quad \hbox{ and } \quad \overline m_{\delta_\eps/L_\eps,2}
    ({b}/({\eps L_\eps}) -{\delta_\eps}/{L_\eps}) \to n_0.    
  \end{align}
  Using the inequality in \eqref{use:ineq} and recalling
  \eqref{Fepsmin}, we obtain the desired result.
\end{proof}

\bibliographystyle{plain}
\bibliography{../../nonlin,../../mura,../../stat}

\end{document}